\documentclass[12pt]{article}

\usepackage{amsmath} 
\usepackage{amsthm}
\usepackage{amssymb}
\usepackage{latexsym}
\usepackage{xcolor}
\usepackage{xspace}
\usepackage[margin=1in]{geometry}

\numberwithin{equation}{section}
\newtheorem{theorem}[equation]{Theorem}
\newtheorem{proposition}[equation]{Proposition}
\newtheorem{definition}[equation]{Definition}

\newtheorem{lemma}[equation]{Lemma}

\title{Mapping properties of weakly singular periodic volume  potentials  in Roumieu classes}

\author{M.~Dalla Riva\thanks{Department of Mathematics,
The University of Tulsa,
800 S Tucker Dr,
Tulsa, Oklahoma 74104, USA} $^{\,,}$\thanks{Department of Mathematics, Aberystwyth University, Ceredigion SY23 3BZ, Wales, UK.}\ , M.~Lanza de Cristoforis\thanks{Dipartimento di Matematica ``Tullio Levi-Civita'', Universit\`a degli Studi di Padova, Via Trieste 63, Padova 35121, Italy}\ ,  and P.~Musolino{$^\dag$}}

\date{$\ $}

\begin{document}

\maketitle





\begin{abstract}
        The analysis of the dependence of integral operators on perturbations plays an important role in the study of inverse problems and of perturbed boundary value problems. In this paper we focus on the mapping properties of the volume  potentials with weakly singular periodic kernels. Our main result is to prove that the map which takes a density function and a periodic kernel to a (suitable restriction of the) volume potential is bilinear and continuous with values in a Roumieu class of analytic functions. Such result extends to the periodic case some previous results obtained by the authors for non periodic potentials and it is motivated by the study of perturbation problems for the solutions of boundary value problems for elliptic differential equations in periodic domains.
\end{abstract}



\section{Introduction}
\label{introd}
 This paper deals with the mapping properties of certain integral operators which arise in periodic potential theory. More precisely,  the authors continue the work  of \cite{DaLaMu15, La05}  where it has been shown that    volume  potentials associated to a parameter dependent analytic family of weakly singular kernels depend real-analyt\-ically  upon the density   function  and on the parameter.
The results of  \cite{DaLaMu15} have found application in the analysis of volume potentials corresponding to an analytic family  of fundamental solutions of second order differential operators with constant coefficients. Here, instead, the authors   extend the results of \cite{DaLaMu15} to the case where the kernels are (spatially) periodic functions and thus provide a useful tool to study the integral operators which appear in the analysis of periodic boundary value problems. \par

Several authors have studied the dependence of integral operators upon perturbations, with particular attention to layer and volume potentials associated to  elliptic differential equations. For example, Potthast has proved in  \cite{Po94,Po96a,Po96b} Fr\'echet differentiability results for the dependence of layer potentials for the Helmholtz equation upon the support of integration in the framework of Schauder spaces. In the frame of inverse problems, we mention the works by Charalambopoulos \cite{Ch95}, Haddar and Kress \cite{HaKr04}, Hettlich \cite{He95}, Kirsch \cite{Ki93}, and Kress and P\"aiv\"arinta \cite{KrPa99}. Moreover, Costabel and Le Lou\"er \cite{CoLe12a, CoLe12b, Le12} obtained analogous results in the framework of Sobolev spaces on Lipschitz  domains.

The authors of the present paper and collaborators have developed a method based on functional analysis and on potential theory to prove analyticity results for the solution of boundary value problems upon perturbations of the domain and of the data. Indeed, classical potential theory enables to represent the solution of a boundary value problem in terms of integral operators. Therefore, in order to study the behavior of the solution upon perturbation, one needs to investigate perturbation results for layer and volume potentials. Thus 
 \cite{LaRo04, LaRo08} have analyzed  the layer potentials associated to the Laplace and Helmholtz equations, whereas   \cite{DaLa10}  has investigated the case of layer potentials corresponding to   second order complex constant coefficient elliptic differentials operators,  and   \cite{LaMu11} has considered  a periodic analog.
 
 The aim of this paper is to analyze the behavior  of  volume potentials corresponding to periodic kernels.  More precisely, here we prove that the bilinear map which takes a weakly singular periodic kernel and a density function to the corresponding periodic volume potential is continuous with values in a Roumieu class of analytic functions (see Theorems \ref{qropo} and \ref{qropoo}). As shown by Preciso \cite{Pr98, Pr00}, such a space of analytic functions is the convenient choice in order the ensure analyticity results for the composition operators in Schauder spaces (see also \cite[\S 1, 6]{DaLaMu15}). We note that, by replacing the weakly singular periodic kernels by periodic analogs of the fundamental solution of elliptic differential operators with constant coefficients,  the results of the present paper  allow to analyze perturbation problems for periodic boundary value problems for non-homogeneous differential equations by a potential theoretic method (cf.~{\it e.g.}, \cite[\S 5]{DaLaMu15}). The reduction of periodic problems to integral equations  is a powerful tool to investigate singular perturbation problems in periodic domains. As an example, we mention the asymptotic extension  method of Ammari and Kang \cite{AmKa07}, Ammari, Kang, and Lee \cite{AmKaLe09}, and the functional analytic approach of the authors  \cite{DaMu13, LaMu14}. As it is well known, periodic boundary value problems have a large variety of applications. Here we refer for example to Ammari and Kang \cite[Chs.~2, 8]{AmKa07}, Kapanadze, Mishuris, and Pesetskaya \cite{KaMiPe15, KaMiPe15b}, Milton \cite[Ch.~1]{Mi02}, Mishuris and Slepyan \cite{MiSl14}, Mityushev,  Pesetskaya, and Rogosin \cite{MiPeRo08}, and Movchan, Movchan,  and Poulton \cite{MoMoPo02}. In particular, such problems are relevant in the  computation of effective properties of composite materials, which in turn can be justified by the homogenization theory (cf.~\textit{e.g.}, Allaire \cite[Ch. 1]{Al02}, Bakhvalov and Panasenko \cite{BaPa89}, Bensoussan, Lions, and Papanicolaou \cite[Ch. 1]{BeLiPa78}).

The paper is organized as follows. In section \ref{nota}, we introduce some basic notation. In section \ref{qvopoke} we prove our main Theorems \ref{qropo} and \ref{qropoo}, where we consider volume potentials with periodic kernel taken in a suitable weighted space of analytic functions and density function belonging to a Roumieu class. We show that the map which takes the pair consisting of the kernel and of the density  to a suitable restriction of the corresponding volume potential is bilinear and continuous with values in a Roumieu class. In particular, in Theorem \ref{qropo} we consider the restriction of the volume potential to a periodic infinite  union of bounded sets, while in Theorem \ref{qropoo} we study the restriction to the complement of a periodic infinite union of bounded sets.  In the last Section \ref{Sq}, we present a class of  kernels which are of the type considered in our main Theorems \ref{qropo} and \ref{qropoo} and which play an important role in the treatment of periodic boundary value problems. Such a class consists of the $q$-periodic analogs of the fundamental solution of elliptic partial differential operators of second order.

\section{Notation}\label{nota}

We  denote the norm on 
a   normed space ${\mathcal X}$ by $\|\cdot\|_{{\mathcal X}}$. Let 
${\mathcal X}$ and ${\mathcal Y}$ be normed spaces. We endow the  
space ${\mathcal X}\times {\mathcal Y}$ with the norm defined by 
$\|(x,y)\|_{{\mathcal X}\times {\mathcal Y}}\equiv \|x\|_{{\mathcal X}}+
\|y\|_{{\mathcal Y}}$ for all $(x,y)\in  {\mathcal X}\times {\mathcal 
Y}$, while we use the Euclidean norm for ${\mathbb{R}}^{n}$.
 For 
standard definitions of Calculus in normed spaces, we refer to 
Deimling~\cite{De85}. The symbol ${\mathbb{N}}$ denotes the 
set of natural numbers including $0$.    Let 
${\mathbb{D}}\subseteq {\mathbb {R}}^{n}$. Then $\mathrm{cl}{\mathbb{D}}$ 
denotes the 
closure of ${\mathbb{D}}$, and $\partial{\mathbb{D}}$ denotes the boundary of ${\mathbb{D}}$. 
The symbol
$| \cdot|$ denotes the Euclidean modulus   in
${\mathbb{R}}^{n}$ or in ${\mathbb{C}}$. For all $R\in]0,+\infty[$, $ x\in{\mathbb{R}}^{n}$, 
$x_{j}$ denotes the $j$-th coordinate of $x$, and  
 ${\mathbb{B}}_{n}( x,R)$ denotes the ball $\{
y\in{\mathbb{R}}^{n}:\, | x- y|<R\}$. 
Let $\Omega$ be an open 
subset of ${\mathbb{R}}^{n}$. The space of $m$ times continuously 
differentiable complex-valued functions on $\Omega$ is denoted by 
$C^{m}(\Omega,{\mathbb{C}})$, or more simply by $C^{m}(\Omega)$. 
Let $f\in \left(C^{m}(\Omega)\right) $. Then   $Df$ denotes the gradient of $f$. 
Let  $\eta\equiv
(\eta_{1},{...} ,\eta_{n})\in{\mathbb{N}}^{n}$, $|\eta |\equiv
\eta_{1}+{...} +\eta_{n}  $. Then $D^{\eta} f$ denotes
$\frac{\partial^{|\eta|}f}{\partial
x_{1}^{\eta_{1}}{...}\partial x_{n}^{\eta_{n}}}$.    The
subspace of $C^{m}(\Omega )$ of those functions $f$ whose derivatives $D^{\eta }f$ of
order $|\eta |\leq m$ can be extended with continuity to 
$\mathrm{cl}\Omega$  is  denoted $C^{m}(
\mathrm{cl}\Omega )$. 
The
subspace of $C^{m}(\mathrm{cl}\Omega ) $  whose
functions have $m$-th order derivatives that are
H\"{o}lder continuous  with exponent  $\alpha\in
]0,1]$ is denoted $C^{m,\alpha} (\mathrm{cl}\Omega )$  
(cf.~\textit{e.g.},~Gilbarg and 
Trudinger~\cite{GiTr83}).  Let 
${\mathbb{D}}\subseteq {\mathbb{C}}^{n}$. Then $C^{m
,\alpha }(\mathrm{cl}\Omega ,{\mathbb{D}})$ denotes
$\left\{f\in \left( C^{m,\alpha} (\mathrm{cl}\Omega )\right)^{n}:\ f(
\mathrm{cl}\Omega )\subseteq {\mathbb{D}}\right\}$. The subspace of $C^{m}(\mathrm{cl}\Omega ) $ of those functions $f$ such that $f_{|{\mathrm{cl}}(\Omega\cap{\mathbb{B}}_{n}(0,R))}\in
C^{m,\alpha}({\mathrm{cl}}(\Omega\cap{\mathbb{B}}_{n}(0,R)))$ for all $R\in]0,+\infty[$ is denoted $C^{m,\alpha}_{{\mathrm{loc}}}(\mathrm{cl}\Omega ) $. \par
Now let $\Omega $ be a bounded
open subset of  ${\mathbb{R}}^{n}$. Then $C^{m}(\mathrm{cl}\Omega )$ 
and $C^{m,\alpha }({\mathrm{cl}}
\Omega )$ are endowed with their usual norm and are well known to be 
Banach spaces  (cf.~\textit{e.g.}, Troianiello~\cite[\S 1.2.1]{Tr87}). For the definition of a bounded open Lipschitz subset of ${\mathbb{R}}^{n}$, we refer for example to Ne\v{c}as~\cite[\S 1.3]{Ne67}. We say that a bounded open subset $\Omega$ of ${\mathbb{R}}^{n}$ is of class 
$C^{m}$ or of class $C^{m,\alpha}$, if it is a 
manifold with boundary imbedded in 
${\mathbb{R}}^{n}$ of class $C^{m}$ or $C^{m,\alpha}$, respectively
 (cf.~\textit{e.g.}, Gilbarg and Trudinger~\cite[\S 6.2]{GiTr83}). 
We denote by 
$
\nu_{\Omega}
$
the outward unit normal to $\partial\Omega$.  For standard properties of functions 
in Schauder spaces, we refer the reader to Gilbarg and 
Trudinger~\cite{GiTr83} and to Troianiello~\cite{Tr87}
(see also \cite[\S 2]{LaRo04}).

\par
We denote by $d\sigma$ the area element of a manifold imbedded in ${\mathbb{R}}^{n}$. We retain the standard notation for the Lebesgue spaces. \par

We note that 
throughout the paper `analytic'   means always  `real analytic'. For the 
definition and properties of analytic operators, we refer to Deimling~\cite[\S 15]{De85}.  \par

 If $\Omega$ is an open subset of ${\mathbb{R}}^{n}$, $k\in {\mathbb{N}}$, $\beta\in]0,1]$, we set
\[
C^{k}_{b}({\mathrm{cl}}\Omega)\equiv
\{
u\in C^{k}({\mathrm{cl}}\Omega):\,
D^{\gamma}u\ {\mathrm{is\ bounded}}\ \forall\gamma\in {\mathbb{N}}^{n}\
{\mathrm{such\ that}}\ |\gamma|\leq k
\}\,,
\]
and we endow $C^{k}_{b}({\mathrm{cl}}\Omega)$ with its usual  norm
\[
\|u\|_{ C^{k}_{b}({\mathrm{cl}}\Omega) }\equiv
\sum_{|\gamma|\leq k}\sup_{x\in {\mathrm{cl}}\Omega }|D^{\gamma}u(x)|\qquad\forall u\in C^{k}_{b}({\mathrm{cl}}\Omega)\,. 
\]
Then we set
\[
C^{k,\beta}_{b}({\mathrm{cl}}\Omega)\equiv
\{
u\in C^{k,\beta}({\mathrm{cl}}\Omega):\,
D^{\gamma}u\ {\mathrm{is\ bounded}}\ \forall\gamma\in {\mathbb{N}}^{n}\
{\mathrm{such\ that}}\ |\gamma|\leq k
\}\,,
\]
and we endow $C^{k,\beta}_{b}({\mathrm{cl}}\Omega)$ with its usual  norm
\[
\|u\|_{ C^{k,\beta}_{b}({\mathrm{cl}}\Omega) }\equiv
\sum_{|\gamma|\leq k}\sup_{x\in {\mathrm{cl}}\Omega }|D^{\gamma}u(x)|
+\sum_{|\gamma| = k}|D^{\gamma}u: {\mathrm{cl}}\Omega |_{\beta}
\qquad\forall u\in C^{k,\beta}_{b}({\mathrm{cl}}\Omega)\,,
\]
where $|D^{\gamma}u: {\mathrm{cl}}\Omega |_{\beta}$ denotes the $\beta$-H\"{o}lder constant of $D^{\gamma}u$.

We now fix once for all a natural number 
\[
n\in {\mathbb{N}}\setminus\{0,1\}
\]
and a periodicity cell
\[
Q\equiv\Pi_{j=1}^{n}]0,q_{jj}[\,,
\]
with
 \[
 (q_{11},{...},q_{nn})\in]0,+\infty[^{n}\, .
\]
Then we denote by $q$ the diagonal matrix
\begin{equation}\label{diag}
q\equiv \left(
\begin{array}{cccc}
q_{11} &   0 & {...} & 0   
\\
0          &q_{22} &{...} & 0
\\
{...} & {...} & {...} & {...}  
\\
0& 0 & {...} & q_{nn}
\end{array}\right)
\end{equation}
and by $m_{n}(Q)$ the $n$-dimensional  measure of the fundamental cell $Q$. Clearly, 
\[
q {\mathbb{Z}}^{n}\equiv  \{qz:\,z\in{\mathbb{Z}}^{n}\}
\]
is the set of vertices of a periodic subdivision of ${\mathbb{R}}^{n}$ corresponding to the fundamental cell $Q$.  Let ${\mathrm{cl}}\Omega\subseteq Q$. Then we introduce the periodic domains
\[
{\mathbb{S}}[\Omega] \equiv \bigcup_{z\in {\mathbb{Z}}^{n}}\left(
qz+\Omega
\right)\,,
\qquad
{\mathbb{S}}[\Omega]^{-} \equiv
{\mathbb{R}}^{n}\setminus{\mathrm{cl}}{\mathbb{S}} [\Omega]\,.
\]
A function $u^{i}$  from ${\mathrm{cl}}{\mathbb{S}}[\Omega] $ to ${\mathbb{R}}$
is $q$-periodic if
\[
u^{i}(x+q_{hh}e_{h})=u^{i}(x)\qquad\forall x\in {\mathrm{cl}}{\mathbb{S}}[\Omega] \,,
\]
for all $h\in\{1,{...},n\}$, and  a function $u^{o}$  from ${\mathrm{cl}}{\mathbb{S}}[\Omega]^{-}$ to ${\mathbb{R}}$
is $q$-periodic if
\[
u^{o}(x+q_{hh}e_{h})=u^{o}(x)\qquad\forall x\in {\mathrm{cl}}{\mathbb{S}}[\Omega]^{-}\,,
\]
for all $h\in\{1,{...},n\}$.   Here $\{e_{1}$,{...}, $e_{n}\}$ denotes the canonical basis of ${\mathbb{R}}^{n}$.

We now introduce some functional spaces of $q$-periodic functions. If $k\in {\mathbb{N}}\cup\{  \infty  \}$ and $\beta\in]0,1]$, then we set
\[
C^{k}_{q}({\mathrm{cl}}{\mathbb{S}}[\Omega] )
\equiv\left\{
u\in C^{k}_{b}({\mathrm{cl}}{\mathbb{S}}[\Omega] ):\,
u\ {\mathrm{is}}\ q-{\mathrm{periodic}}
\right\}\,,
\]
which we regard as a Banach subspace of $C^{k}_{b}({\mathrm{cl}}{\mathbb{S}}[\Omega] )$, and 
\[
C^{k,\beta}_{q}({\mathrm{cl}}{\mathbb{S}}[\Omega] )
\equiv\left\{
u\in C^{k,\beta}_{b}({\mathrm{cl}}{\mathbb{S}}[\Omega] ):\,
u\ {\mathrm{is}}\ q-{\mathrm{periodic}}
\right\}\,,
\]
which we regard as a Banach subspace of $C^{k,\beta}_{b}({\mathrm{cl}}{\mathbb{S}}[\Omega] )$, and 
\[
C^{k}_{q}({\mathrm{cl}}{\mathbb{S}}[\Omega]^{-})
\equiv\left\{
u\in C^{k}_{b}({\mathrm{cl}}{\mathbb{S}}[\Omega]^{-}):\,
u\ {\mathrm{is}}\ q-{\mathrm{periodic}}
\right\}\,,
\]
which we regard as a Banach subspace of $C^{k}_{b}({\mathrm{cl}}{\mathbb{S}}[\Omega]^{-})$, and
\[
C^{k,\beta}_{q}({\mathrm{cl}}{\mathbb{S}}[\Omega]^{-} )
\equiv\left\{
u\in C^{k,\beta}_{b}({\mathrm{cl}}{\mathbb{S}}[\Omega]^{-} ):\,
u\ {\mathrm{is}}\ q-{\mathrm{periodic}}
\right\}\,,
\]
which we regard as a Banach subspace of $C^{k,\beta}_{b}({\mathrm{cl}}{\mathbb{S}}[\Omega]^{-})$.  We also set
\[
C^{k}_{q}({\mathbb{R}}^{n}\setminus q{\mathbb{Z}}^{n} )\equiv \left\{
u\in C^{k}({\mathbb{R}}^{n}\setminus q{\mathbb{Z}}^{n} ):\,
u(x+qz)=u(x) \ \forall x \in \mathbb{R}^n \setminus q\mathbb{Z}^n\, , \forall z \in \mathbb{Z}^n
\right\}\,.
\]

Next, we turn to introduce the Roumieu classes. For all bounded open subsets $\Omega$ 
of ${\mathbb{R}}^{n}$ and $\rho\in]0,+\infty[$, we set
\[
C^{0}_{\omega,\rho}({\mathrm{cl}}\Omega)\equiv
\biggl\{u\in C^{\infty}({\mathrm{cl}}\Omega):\,
\sup_{\beta\in{\mathbb{N}}^n}\frac{\rho^{|\beta |}}{|\beta |!}\|D^\beta 
u\|_{C^{0}({\mathrm{cl}}\Omega)}<+\infty 
\biggr\}\,,
\]
and
\[
\|u\|_{C^{0}_{\omega,\rho}({\mathrm{cl}}\Omega)}\equiv
\sup_{\beta\in{\mathbb{N}}^n}\frac{\rho^{|\beta |}}{|\beta |!}\|D^\beta 
u\|_{C^{0}({\mathrm{cl}}\Omega)}\qquad\forall u\in
C^{0}_{\omega,\rho}({\mathrm{cl}}\Omega)\,.
\]
As it is well known, the Roumieu class $\left(
C^{0}_{\omega,\rho}({\mathrm{cl}}\Omega),
\|\cdot\|_{C^{0}_{\omega,\rho}({\mathrm{cl}}\Omega)}
\right)$ is a Banach space. \par

 If ${\mathrm{cl}}\Omega\subseteq Q$ and $\rho>0$,  then we set
\[
C^{0}_{q,\omega,\rho}({\mathrm{cl}}{\mathbb{S}}[\Omega])
\equiv
\biggl\{u\in C^{\infty}_{q}({\mathrm{cl}}{\mathbb{S}}[\Omega]):\,
\sup_{\beta\in{\mathbb{N}}^n}\frac{\rho^{|\beta |}}{|\beta |!}\|D^\beta 
u\|_{C^{0}({\mathrm{cl}}\Omega)}<+\infty 
\biggr\}\,,
\]
and
\[
\|u\|_{C^{0}_{q,\omega,\rho}({\mathrm{cl}}{\mathbb{S}}[\Omega])}\equiv
\sup_{\beta\in{\mathbb{N}}^n}\frac{\rho^{|\beta |}}{|\beta |!}\|D^\beta 
u\|_{C^{0}(\mathrm{cl}\Omega)}\qquad\forall u\in
C^{0}_{q,\omega,\rho}({\mathrm{cl}}{\mathbb{S}}[\Omega])\,.
\]
Similarly, we set
\[
C^{0}_{q,\omega,\rho}({\mathrm{cl}}{\mathbb{S}}[\Omega]^{-})
\equiv
\biggl\{u\in C^{\infty}_{q}({\mathrm{cl}}{\mathbb{S}}[\Omega]^{-}):\,
\sup_{\beta\in{\mathbb{N}}^n}\frac{\rho^{|\beta |}}{|\beta |!}\|D^\beta 
u\|_{C^{0}({\mathrm{cl}}Q\setminus\Omega)}<+\infty 
\biggr\}\,,
\]
and
\[
\|u\|_{C^{0}_{q,\omega,\rho}({\mathrm{cl}}{\mathbb{S}}[\Omega]^{-})}\equiv
\sup_{\beta\in{\mathbb{N}}^n}\frac{\rho^{|\beta |}}{|\beta |!}\|D^\beta 
u\|_{C^{0}({\mathrm{cl}}Q\setminus\Omega)}\qquad\forall u\in
C^{0}_{q,\omega,\rho}({\mathrm{cl}}{\mathbb{S}}[\Omega]^{-})\,.
\]
The  periodic  Roumieu class  $\left(
C^{0}_{q,\omega,\rho}({\mathrm{cl}}{\mathbb{S}}[\Omega] ),
\|\cdot\|_{C^{0}_{q,\omega,\rho}({\mathrm{cl}}{\mathbb{S}}[\Omega] )}
\right)$ 
and the  periodic  Rou\-mieu class
$\left(
C^{0}_{q,\omega,\rho}({\mathrm{cl}}{\mathbb{S}}[\Omega]^{-}),
\|\cdot\|_{C^{0}_{q,\omega,\rho}({\mathrm{cl}}{\mathbb{S}}[\Omega]^{-})}
\right)$ 
are Banach spaces.


\section{Periodic volume potentials corresponding to general kernels in Roumieu classes}
\label{qvopoke}

We set 
\[
\tilde{Q}\equiv\Pi_{j=1}^{n}]-q_{jj}/2,q_{jj}/2[\,,
\]
and
\[
{\mathbb{Z}}_{Q}\equiv 
\left\{
z\in {\mathbb{Z}}^{n}:\, qz\in \partial Q
\right\}\,.
\]
Obviously, ${\mathbb{Z}}_{Q}$ has $2^{n}$ elements, and 
\[
Q\subseteq\bigcup_{z\in  {\mathbb{Z}}_{Q} }
(qz+{\mathrm{cl}}\tilde{Q} )\,,
\]
and
\[
m_{n}\left(
Q\setminus
\bigcup_{z\in  {\mathbb{Z}}_{Q} }
(qz+ \tilde{Q} )
\right)=0\,.
\]
By exploiting the absolute continuity of a measure associated to an integrable function, we can prove the following elementary lemma (see also   \cite[\S 3]{DaLaMu15}).
\begin{lemma}
Let $f\in L^{1}_{{\mathrm{loc}}}({\mathbb{R}}^{n})$ be $q$-periodic. Then for all $\epsilon>0$ there exists $\delta>0$ such that
\[
\int_{x+E}|f|\,dx\leq \epsilon\,,
\]
for all measurable subsets $E$ of $Q$ such that $m_{n}(E)\leq\delta$ and for all $x\in  {\mathbb{R}}^{n}$.
\end{lemma}

Then we have the following.
\begin{lemma}
\label{acper}
Let $\lambda\in ]0,n[$. Let $h\in C^{0}_{q}({\mathbb{R}}^{n}\setminus q{\mathbb{Z}}^{n} )$. Let
\begin{equation}
\label{acper1}
\sup_{x\in \tilde{Q}\setminus\{0\}} |h(x)|\,|x|^{\lambda}<+\infty\,.
\end{equation}
Then for all $\epsilon>0$ there exists $\delta>0$ such that
\[
\int_{E}|h(x-y)|\,dy\leq\epsilon\,,
\]
for all measurable subsets $E$ of $Q$ such that $m_{n}(E)\leq\delta$ and for all $x\in  {\mathbb{R}}^{n}$.
\end{lemma}
\begin{proof} We first prove that $h_{|Q}\in L^{1}(Q)$. Condition (\ref{acper1}) implies the summability of $h$ in $\tilde{Q}$. Moreover, by the periodicity of $h$ we have 
\begin{equation}
\label{acper2}
\int_{Q}|h(\eta)|\,d\eta\leq \sum_{z\in {\mathbb{Z}}_{Q} }
\int_{qz+\tilde{Q} }|h(\eta)|\,d\eta= 2^{n}\int_{ \tilde{Q} }|h(\eta)|\,d\eta\,.
\end{equation}
Hence, the summability of $h$ in $Q$ follows by (\ref{acper2}).  Since $h_{|Q}\in L^{1}(Q)$ and $h$ is $q$-periodic, we have $h\in L^{1}_{{\mathrm{loc}}}({\mathbb{R}}^{n})$. Then   $h(-\cdot)$ is also locally integrable and $q$-periodic. Therefore,   there exists $\delta>0$ such that
\[
\int_{ E}|h(x-y)|\,dy=
\int_{-x+E}|h(-\eta)|\,d\eta\leq \epsilon\qquad\forall x\in {\mathbb{R}}^{n}\,,
\]
for all measurable subsets $E$ of $Q$ such that $m_{n}(E)\leq\delta$. \end{proof} 

\vspace{\baselineskip}

Next we introduce the following class of $q$-periodic functions which are singular in the punctured periodicity cell  $\tilde{Q}\setminus\{0\}$. 
\begin{definition}
\label{qa0r}
Let $\lambda\in]0,+\infty[$.  Then we denote by 
$A^{0}_{q,\lambda} $ the set of the functions
$h\in {C^{0}_{q}({\mathbb{R}}^{n} \setminus q{\mathbb{Z}}^{n})}$ such that
\[
\sup_{x\in \tilde{Q}\setminus\{0\} }
|h(x)|\,|x|^{\lambda}<+\infty\,, 
\]
and we set
\[
\|h\|_{A^{0}_{q,\lambda} }\equiv
 \sup_{x\in \tilde{Q}\setminus\{0\} }
|h(x)|\,|x|^{\lambda}
\qquad\forall h\in A^{0}_{q,\lambda} 
\,.
\]
\end{definition}
One can readily verify that $(A^{0}_{q,\lambda} , \|\cdot\|_{A^{0}_{q,\lambda} })$ is a Banach space. 
\begin{proposition}
\label{qopoi}
Let $\lambda\in]0,n[$. Let $\Omega$ be a bounded open subset of 
${\mathbb{R}}^{n}$ such that ${\mathrm{cl}}\Omega\subseteq Q$. 
Then the following statements hold.
\begin{enumerate}
\item[(i)] If $(h,\varphi)\in  A^{0}_{q,\lambda} \times L^{\infty}(\Omega)$ and if $x\in {\mathbb{R}}^{n}$, then the function from $\Omega$  to ${\mathbb{R}}$ which takes $y\in\Omega$ to $h(x-y)\varphi(y)$ is integrable.
\item[(ii)]  If $(h,\varphi)\in  A^{0}_{q,\lambda} \times L^{\infty}(\Omega)$, then the function  ${\mathcal{P}}^{+}[h,\varphi]$ from ${\mathbb{R}}^{n}$ to ${\mathbb{R}}$ which takes
$x\in {\mathbb{R}}^{n}$  to 
\[
{\mathcal{P}}^{+}[h,\varphi](x)\equiv\int_{ \Omega }h(x-y)\varphi (y)\,dy\qquad\forall x\in{\mathbb{R}}^{n}\,,
\]
is continuous and $q$-periodic. 
\item[(iii)] ${\mathcal{P}}^{+}[h,\varphi]$ is bounded and 
\[
\|{\mathcal{P}}^{+}[h,\varphi]\|_{L^\infty(
{\mathbb{R}}^{n})}
\leq 2^{n} \int_{\tilde{Q} }|y|^{-\lambda}\,dy
\|h\|_{A^{0}_{q,\lambda} } 
\|\varphi\|_{L^{\infty}(\Omega)}
 \,,
\]
for all $(h,\varphi)\in A^{0}_{q,\lambda} \times L^{\infty}(\Omega)$.
\item[(iv)] If $(h,\varphi)\in  A^{0}_{q,\lambda} \times L^{\infty}(Q\setminus {\mathrm{cl}} \Omega)$ and if $x\in   {\mathbb{R}}^{n}$, then the function from $Q\setminus {\mathrm{cl}} \Omega$  to ${\mathbb{R}}$ which takes $y\in Q\setminus {\mathrm{cl}} \Omega$ to $h(x-y)\varphi(y)$ is integrable.
\item[(v)]  If $(h,\varphi)\in  A^{0}_{q,\lambda} \times L^{\infty}(Q\setminus {\mathrm{cl}} \Omega)$, then the function  ${\mathcal{P}}^{-}[h,\varphi]$ from ${\mathbb{R}}^{n}$ to ${\mathbb{R}}$ which takes
$x\in {\mathbb{R}}^{n}$  to 
\[
{\mathcal{P}}^{-}[h,\varphi](x)\equiv\int_{ 
Q\setminus {\mathrm{cl}} \Omega
 }h(x-y)\varphi (y)\,dy\qquad\forall x\in{\mathbb{R}}^{n}\,,
\]
is continuous and $q$-periodic. 
\item[(vi)] ${\mathcal{P}}^{-}[h,\varphi]$ is bounded and 
\[
\|{\mathcal{P}}^{-}[h,\varphi]\|_{L^\infty(
{\mathbb{R}}^{n})}
\leq 2^{n} \int_{\tilde{Q} }|y|^{-\lambda}\,dy
\|h\|_{A^{0}_{q,\lambda} } 
\|\varphi\|_{L^{\infty}(Q\setminus {\mathrm{cl}} \Omega)}
 \,,
\]
for all $(h,\varphi)\in A^{0}_{q,\lambda} \times L^{\infty}(Q\setminus {\mathrm{cl}} \Omega)$.
\end{enumerate}
\end{proposition}
\begin{proof} If $(h,\varphi)\in A^{0}_{q,\lambda} \times L^{\infty}(\Omega)$, then we have 
\begin{equation}
\label{qopoi5a}
|h(x-y)\varphi(y)|\leq \sum_{z\in {\mathbb{Z}}_{Q}}
\chi_{qz+\tilde{Q} }(y) |h(x-y)|   \|\varphi\|_{L^{\infty}(\Omega)}
\qquad{\mathrm{for\ a.a.}}\ y\in \Omega\, ,
\end{equation}
for all $x \in \mathbb{R}^n$. Since $h$ is $q$-periodic and $\tilde{Q}$ is a $q$-periodicity cell, we have 
\begin{eqnarray}
\label{qopoi5aa}
\lefteqn{
\int_{qz+\tilde{Q}}|h(x-y)|\,dy
=
 \int_{x+qz+\tilde{Q}}|h(x-y)|\,dy
}
\\ \nonumber
&&\qquad
=\int_{qz+\tilde{Q}}|h( -y)|\,dy
=\int_{ \tilde{Q}}|h( -y)|\,dy
\\ \nonumber
&&\qquad
 =\int_{ \tilde{Q}}|h( -y)|\, |-y|^{\lambda}\frac{1}{ | y|^{\lambda}}\,dy
\leq
 \int_{ \tilde{Q}}\frac{1}{ | y|^{\lambda}}\,dy
\|h\|_{A^{0}_{q,\lambda} } 
<+\infty\qquad\forall x\in {\mathbb{R}}^{n}\,,
\end{eqnarray}
for all $z\in\mathbb{Z}_Q$. Since ${\mathbb{Z}}_{Q}$ has $2^{n}$ elements, statement (i) follows. Next we prove the continuity of ${\mathcal{P}}^{+}[h,\varphi]$. Let $x_{0}\in {\mathbb{R}}^{n}$. Let $\epsilon>0$. By Lemma \ref{acper}, there exists $\delta\in ]0,\frac{1}{2}[$ such that
\begin{equation}
\label{qopoi6}
\int_{(\xi+\delta\tilde{Q})\cap Q}|h(x-y)|\,dy\leq \epsilon/2^{n+1}
\qquad\forall (\xi,x)\in   {\mathbb{R}}^{n}\times  {\mathbb{R}}^{n}\,.  
\end{equation}
Since  $\delta\in ]0,1/2[$, the set
\begin{equation}
\label{qopoi7}
{\mathbb{Z}}(x_{0},\delta)\equiv
\left\{
z\in {\mathbb{Z}}^{n}:\, 
(x_{0}+qz+\delta{\mathrm{cl}}\tilde{Q})\cap{\mathrm{cl} }Q\neq\emptyset
\right\}
\end{equation}
has at most $2^{n}$ elements. Since the function $|h(x-y)|$ is continuous on the compact  set
\begin{equation}
\label{qopoi8}
\left(x_{0}+\frac{\delta}{2}{\mathrm{cl}}\tilde{Q}\right)
\times
\left(
{\mathrm{cl}}Q\setminus
\left(
\bigcup_{z\in {\mathbb{Z}}(x_{0},\delta)
}
(x_{0}+qz+\delta \tilde{Q})
\right)
\right)\,,
\end{equation}
the function $|h(x-y)|$ has a maximum $\gamma(x_{0},\delta)$ on such a set. Then by (\ref{qopoi6}), we have that
\begin{eqnarray}
\nonumber
\lefteqn{
\left|
{\mathcal{P}}^{+}_{q}[h,\varphi](x)
-
{\mathcal{P}}^{+}_{q}[h,\varphi](x_{0})
\right|
}
\\ \nonumber
&&\qquad
\leq
\biggl|\biggr.
\int_{
\Omega \setminus
\bigcup_{z\in {\mathbb{Z}}(x_{0},\delta)
}
(x_{0}+qz+\delta{\mathrm{cl}}\tilde{Q})
}h(x-y)\varphi(y)\,dy
\\ \nonumber
&&\qquad\qquad\qquad
-
\int_{
\Omega \setminus
\bigcup_{z\in {\mathbb{Z}}(x_{0},\delta)
}
(x_{0}+qz+\delta{\mathrm{cl}}\tilde{Q})
}h(x_{0}-y)\varphi(y)\,dy
\biggl.\biggr| 
\\ \nonumber
&& \qquad\qquad
+\int_{
\bigcup_{z\in {\mathbb{Z}}(x_{0},\delta)
}
(x_{0}+qz+\delta{\mathrm{cl}}\tilde{Q})
\cap Q
}|h(x-y)|\,dy\|\varphi\|_{L^{\infty}(\Omega)}
\\ \nonumber
&&\qquad\qquad\qquad
+\int_{
\bigcup_{z\in {\mathbb{Z}}(x_{0},\delta)
}
(x_{0}+qz+\delta{\mathrm{cl}}\tilde{Q})
\cap Q
}|h(x_{0}-y)|\,dy\|\varphi\|_{L^{\infty}(\Omega)}
\\ \nonumber
&&\qquad
\leq
\biggl|\biggr.
\int_{
\Omega \setminus
\bigcup_{z\in {\mathbb{Z}}(x_{0},\delta)
}
(x_{0}+qz+\delta{\mathrm{cl}}\tilde{Q})
}h(x-y)\varphi(y)\,dy
\\ \nonumber
&&\qquad\qquad\qquad
-
\int_{
\Omega \setminus
\bigcup_{z\in {\mathbb{Z}}(x_{0},\delta)
}
(x_{0}+qz+\delta{\mathrm{cl}}\tilde{Q})
}h(x_{0}-y)\varphi(y)\,dy
\biggl.\biggr| 
\\ \nonumber
&& \qquad
+
\left(
2^{n}(\epsilon/2^{n+1})
+
2^{n}(\epsilon/2^{n+1})
\right)\|\varphi\|_{L^{\infty}(\Omega)}
\\ \nonumber
&&\qquad
\leq \biggl\{\biggr.
\int_{
\Omega \setminus
\bigcup_{z\in {\mathbb{Z}}(x_{0},\delta)
}
(x_{0}+qz+\delta{\mathrm{cl}}\tilde{Q})
}|h(x-y)-h(x_{0}-y)|\,dy
\\ \nonumber
&&\qquad\qquad\qquad
+2^{n+1}(\epsilon/2^{n+1})
\biggl.\biggr\}\|\varphi\|_{L^{\infty}(\Omega)}
\qquad\forall x\in \left(x_{0}+\frac{\delta}{2}\tilde{Q}\right)
\,,
\end{eqnarray}
and that 
\begin{equation}
\label{qopoi10}
|h(x-y)|\leq \gamma(x_{0},\delta)\,,\qquad
|h(x-y)-h(x_{0}-y)|\leq 2\gamma(x_{0},\delta)\,,
\end{equation}
for all $x\in x_{0}+\frac{\delta}{2}\tilde{Q}$ and $y\in 
 \Omega \setminus
\bigcup_{z\in {\mathbb{Z}}(x_{0},\delta)
}
(x_{0}+qz+\delta{\mathrm{cl}}\tilde{Q})$. Then the above mentioned continuity of $h$  and the Dominated Convergence Theorem imply that 
\[
\limsup_{x\to x_{0}}\left|
{\mathcal{P}}^{+}_{q}[h,\varphi](x)-{\mathcal{P}}^{+}_{q}[h,\varphi](x_{0})
\right|\leq \|\varphi\|_{L^{\infty}(\Omega)}\epsilon\,,
\]
and the continuity of ${\mathcal{P}}^{+}_{q}[h,\varphi]$
at $x_{0}$ follows. Since ${\mathbb{Z}}_{Q}$ has $2^{n}$ elements, inequalities (\ref{qopoi5a}) and (\ref{qopoi5aa})
imply the validity of statement (iii). 

Next we consider statement (iv). If $(h,\varphi)\in  A^{0}_{q,\lambda} \times L^{\infty}(Q\setminus {\mathrm{cl}} \Omega)$, then we have 
\begin{equation}
\label{qopoi10a}
|h(x-y)\varphi(y)|\leq\sum_{z\in {\mathbb{Z}}_{Q}}\chi_{qz+\tilde{Q} }(y)
|h(x-y)|\|\varphi\|_{L^{\infty}(Q\setminus {\mathrm{cl}}\Omega)}
\qquad{\mathrm{for\ a.a.}}\ y\in Q\setminus {\mathrm{cl}}\Omega\,,
\end{equation}
for all $x\in {\mathbb{R}}^{n}$.  By (\ref{qopoi5aa}), we have
\begin{equation}
\label{qopoi11}
\int_{qz+\tilde{Q}}|h(x-y)|\,dy
\leq
 \int_{ \tilde{Q}}\frac{1}{ | y|^{\lambda}}\,dy
\|h\|_{A^{0}_{q,\lambda} }\qquad\forall x\in {\mathbb{R}}^{n}\,.
\end{equation}
Since ${\mathbb{Z}}_{Q}$ has $2^{n}$ elements,  statement (iv) follows. We now consider statement (v). Let $x_{0}\in {\mathbb{R}}^{n}$. Let $\epsilon>0$. Let $\delta\in]0,1/2[$ be as in (\ref{qopoi6}). Since 
  $\delta<1/2$, the set ${\mathbb{Z}}(x_{0},\delta)$ has at most $2^{n}$ elements (see (\ref{qopoi7})).
 Since the function $|h(x-y)|$ is continuous on the compact set in (\ref{qopoi8}), the function  $|h(x-y)|$  has a maximum $\gamma(x_{0},\delta)$ on the set in (\ref{qopoi8}). Then (\ref{qopoi6}) implies that
 \begin{eqnarray}
\nonumber
\lefteqn{
\left|
{\mathcal{P}}^{-}_{q}[h,\varphi](x)
-
{\mathcal{P}}^{-}_{q}[h,\varphi](x_{0})
\right|
}
\\ \nonumber
&&\qquad
\leq
\biggl|\biggr.
\int_{
(Q\setminus {\mathrm{cl}}\Omega) \setminus
\bigcup_{z\in {\mathbb{Z}}(x_{0},\delta)
}
(x_{0}+qz+\delta{\mathrm{cl}}\tilde{Q})
}h(x-y)\varphi(y)\,dy
\\ \nonumber
&&\qquad\qquad\qquad
-
\int_{
(Q\setminus {\mathrm{cl}}\Omega)\setminus
\bigcup_{z\in {\mathbb{Z}}(x_{0},\delta)
}
(x_{0}+qz+\delta{\mathrm{cl}}\tilde{Q})
}h(x_{0}-y)\varphi(y)\,dy
\biggl.\biggr| 
\\ \nonumber
&& \qquad\qquad
+\int_{
\bigcup_{z\in {\mathbb{Z}}(x_{0},\delta)
}
(x_{0}+qz+\delta{\mathrm{cl}}\tilde{Q})
\cap Q
}|h(x-y)|\,dy\|\varphi\|_{L^{\infty}(Q\setminus {\mathrm{cl}}\Omega)}
\\ \nonumber
&&\qquad\qquad\qquad
+\int_{
\bigcup_{z\in {\mathbb{Z}}(x_{0},\delta)
}
(x_{0}+qz+\delta{\mathrm{cl}}\tilde{Q})
\cap Q
}|h(x_{0}-y)|\,dy\|\varphi\|_{L^{\infty}(Q\setminus {\mathrm{cl}}\Omega)}
\\ \nonumber
&&\qquad
\leq
\biggl|\biggr.
\int_{
(Q\setminus {\mathrm{cl}}\Omega) \setminus
\bigcup_{z\in {\mathbb{Z}}(x_{0},\delta)
}
(x_{0}+qz+\delta{\mathrm{cl}}\tilde{Q})
}h(x-y)\varphi(y)\,dy
\\ \nonumber
&&\qquad\qquad\qquad
-
\int_{
(Q\setminus {\mathrm{cl}}\Omega) \setminus
\bigcup_{z\in {\mathbb{Z}}(x_{0},\delta)
}
(x_{0}+qz+\delta{\mathrm{cl}}\tilde{Q})
}h(x_{0}-y)\varphi(y)\,dy
\biggl.\biggr| 
\\ \nonumber
&& \qquad
+
\left(
2^{n}(\epsilon/2^{n+1})
+
2^{n}(\epsilon/2^{n+1})
\right)\|\varphi\|_{L^{\infty}(Q\setminus {\mathrm{cl}}\Omega)}
\\ \nonumber
&&\qquad
\leq \biggl\{\biggr.
\int_{
(Q\setminus {\mathrm{cl}}\Omega) \setminus
\bigcup_{z\in {\mathbb{Z}}(x_{0},\delta)
}
(x_{0}+qz+\delta{\mathrm{cl}}\tilde{Q})
\cap Q
}|h(x-y)-h(x_{0}-y)|\,dy
\\ \nonumber
&&\qquad\qquad\qquad
+2^{n+1}(\epsilon/2^{n+1})
\biggl.\biggr\}\|\varphi\|_{L^{\infty}(Q\setminus {\mathrm{cl}}\Omega)}
\qquad\forall 
x\in \left(
x_{0}+\frac{\delta}{2}\tilde{Q}
\right) \,.
\end{eqnarray}
Then inequality (\ref{qopoi10}), and the continuity of $h$ and the Dominated Convergence Theorem imply that 
\[
\limsup_{x\to x_{0}}\left|
{\mathcal{P}}^{-}_{q}[h,\varphi](x)-{\mathcal{P}}^{-}_{q}[h,\varphi](x_{0})
\right|\leq \|\varphi\|_{L^{\infty}(Q\setminus {\mathrm{cl}}\Omega)}\epsilon\,,
\]
and the continuity of ${\mathcal{P}}^{+}_{q}[h,\varphi]$
at $x_{0}$ follows. Since ${\mathbb{Z}}_{Q}$ has  $2^{n}$ elements,  inequalities (\ref{qopoi10a}) and 
 (\ref{qopoi11}) imply the validity of statement (vi). \end{proof} 

\vspace{\baselineskip}

Next we introduce the following.
\begin{definition}
\label{a1r}
Let $\lambda\in]0,+\infty[$. Then we denote by 
$A^{1}_{q,\lambda} $ the set of the functions
$h\in C^{1}_{q}( {\mathbb{R}}^{n} \setminus q{\mathbb{Z}}^{n})$ such that
\[
h\in A^{0}_{q,\lambda} \,,
\qquad
\frac{\partial h}{\partial x_{j}}\in A^{0}_{q,\lambda+1} 
\qquad\forall j\in\{1,{...},n\}\,,
\]
and we set
\[
\|h\|_{A^{1}_{q,\lambda} }\equiv
\|h\|_{A^{0}_{q,\lambda} }
+
\sum_{j=1}^{n}\left\|
\frac{\partial h}{\partial x_{j} }
\right\|_{A^{0}_{q,\lambda+1} }
\qquad \forall h\in A^{1}_{q,\lambda} \,.
\]
\end{definition}
One can readily verify that $(A^{1}_{q,\lambda} , \|\cdot\|_{A^{1}_{q,\lambda} })$ is a Banach space.

\begin{proposition}
\label{q1poi}
Let $\lambda\in]0,n-1[$. Let $\Omega$ be a bounded open subset of 
${\mathbb{R}}^{n}$ such that ${\mathrm{cl}}\Omega\subseteq Q$. Then the following statements hold. 
\begin{enumerate}
\item[(i)] If $(h,\varphi)\in  A^{1}_{q,\lambda} \times L^{\infty}(\Omega)$ and if $x\in {\mathbb{R}}^{n}$, then the functions from $\Omega$  to ${\mathbb{R}}$ which take  $y\in\Omega$ to $h(x-y)\varphi(y)$ 
and to $\frac{\partial h}{\partial x_{j}}(x-y)\varphi(y)$ for $j\in\{1,{...},n\}$ are
 integrable.
\item[(ii)]  If $(h,\varphi)\in  A^{1}_{q,\lambda} \times L^{\infty}(\Omega)$, then ${\mathcal{P}}^{+}[h,\varphi]\in C^{1}_{q}({\mathbb{R}}^{n})$ and 
\begin{equation}
\label{q1poi1}
\frac{\partial  }{\partial x_{j}}{\mathcal{P}}^{+}[h,\varphi]
={\mathcal{P}}^{+}[\frac{\partial h}{\partial x_{j}},\varphi]
\qquad {\mathrm{in}}\ {\mathbb{R}}^{n}\,,
\end{equation}
for all $j\in\{1,{...}, n\}$. 
\item[(iii)] If $(h,\varphi)\in  A^{1}_{q,\lambda} \times L^{\infty}(Q\setminus{\mathrm{cl}} \Omega)$ and if $x\in {\mathbb{R}}^{n}$, then the functions from $Q\setminus{\mathrm{cl}}\Omega$  to ${\mathbb{R}}$ which take  $y\in Q\setminus{\mathrm{cl}}\Omega$ to $h(x-y)\varphi(y)$ 
and to $\frac{\partial h}{\partial x_{j}}(x-y)\varphi(y)$ for $j\in\{1,{...},n\}$ are
 integrable.
\item[(iv)]  If $(h,\varphi)\in  A^{1}_{q,\lambda} \times L^{\infty}(Q\setminus{\mathrm{cl}}\Omega)$, then ${\mathcal{P}}^{-}[h,\varphi]\in C^{1}_{q}({\mathbb{R}}^{n})$ and 
\[
\frac{\partial  }{\partial x_{j}}{\mathcal{P}}^{-}[h,\varphi]
={\mathcal{P}}^{-}[\frac{\partial h}{\partial x_{j}},\varphi]
\qquad {\mathrm{in}}\ {\mathbb{R}}^{n}\,,
\]
for all $j\in\{1,{...}, n\}$.
\end{enumerate}
\end{proposition}
\begin{proof} Statements (i), (iii) follow by Proposition \ref{qopoi} (i), (iv) applied to the functions $h$, $\frac{\partial h}{\partial x_{j}}$. 

We now prove statement (ii). By Proposition \ref{qopoi} (ii), the two functions ${\mathcal{P}}^{+}[h,\varphi]$ and ${\mathcal{P}}^{+}[\frac{\partial h}{\partial x_{j}},\varphi]$
are continuous and $q$-periodic in ${\mathbb{R}}^{n}$  for all $j\in\{1,{...},n\}$. Then it suffices to prove that $\frac{\partial}{\partial x_{j}}{\mathcal{P}}^{+}[h,\varphi]$ exists  and that (\ref{q1poi1}) holds. In order to prove (\ref{q1poi1}), we proceed by a standard argument.  Let $g\in C^{\infty}(  {\mathbb{R}}  )$ be such that
\[
g(t)=0\qquad\forall t\in ]-\infty,1/8]\,,
\qquad
g(t)=1\qquad\forall t\in [1/4,+\infty[\,.
\]
Then we set
\begin{equation}
\label{q1poi3}
g_{\delta}(y)\equiv\left\{
\begin{array}{ll}
g(\frac{|y-qz|}{\delta})\qquad & \forall y\in qz+{\mathbb{B}}_{n}(0,\delta/4)\,,\qquad\forall z\in {\mathbb{Z}}^{n}\,,
\\
1 &  \forall y\in {\mathbb{R}}^{n}\setminus\bigcup_{z\in {\mathbb{Z}}^{n}}\left(qz+{\mathbb{B}}_{n}(0,\delta/4)\right)\,,
\end{array}
\right.
\end{equation}
for all $\delta\in]0,1[$. 
By definition, $g_{\delta}\in C^{\infty}_{q}({\mathbb{R}}^{n})$ and
\begin{eqnarray}\nonumber
&&g_{\delta}(x)=0\qquad\forall y\in \bigcup_{z\in {\mathbb{Z}}^{n}}\left(qz+{\mathbb{B}}_{n}(0,\delta/8)\right)\,,
\\\nonumber
\label{q1poi3b}
&&g_{\delta}(x)=1\qquad\forall y\in {\mathbb{R}}^{n}\setminus\bigcup_{z\in {\mathbb{Z}}^{n}}\left(qz+{\mathbb{B}}_{n}(0,\delta/4)\right)\,,
\end{eqnarray}
for all $\delta\in]0,1[$. 
Then we set 
\[
u_{\delta}^{+}(x)\equiv \int_{\Omega}
g_{\delta}(x-y)h(x-y)\varphi (y)\,dy\qquad\forall x\in 
{\mathbb{R}}^{n}\,,
\]
for all $\delta\in ]0,1[$. Clearly, $g_{\delta}(x-y)h(x-y)$ is of class $C^{1}$ in the variable $(x,y)\in {\mathbb{R}}^{n}\times {\mathbb{R}}^{n}$. We now show that $u_{\delta}^{+} \in C^{1}_{q}(
{\mathbb{R}}^{n})$ by applying the classical theorem of differentiability  for integrals depending on a parameter. 
 Let $x_{0}\in {\mathbb{R}}^{n}$. Then we have
\begin{eqnarray}
\nonumber
\lefteqn{
\left|
\frac{\partial}{\partial x_{j}}
\left\{
g_{\delta}(x-y)h(x-y)\varphi(y)
\right\}
\right|
}
\\ \nonumber
&&\qquad
\leq
\left|
\frac{\partial}{\partial x_{j}}
\left\{
g_{\delta}(x-y)\right\}h(x-y)\varphi(y)
\right|
+
\left|
g_{\delta}(x-y)\frac{\partial}{\partial x_{j}}
\left\{h(x-y)\right\}\varphi(y)
\right|
\end{eqnarray}
for all $x\in {\mathbb{R}}^{n}$ and for almost all $y \in \Omega$ and for all $\delta\in]0,1[$. By (\ref{q1poi3}), if $z \in \mathbb{Z}^n$ we have
\[
\left|\frac{\partial}{\partial\xi_{j}}g_{\delta}(\xi)\right|
=
\left|
\frac{\partial}{\partial\xi_{j}} \left(g\left(\frac{|\xi-qz|}{\delta}\right)\right)
\right|
=
\left|
g'\left(\frac{|\xi-qz|}{\delta}\right)\frac{\xi_{j}-q_{jj}z_{j}}{|\xi-qz|}\frac{1}{\delta}
\right|
\leq\sup|g'|\frac{1}{\delta}\,,
\]
for all $\xi\in qz+{\mathbb{B}}_{n}(0,\delta/4)$ and for all $\delta\in]0,1[$,
and
\[
\frac{\partial}{\partial\xi_{j}}g_{\delta}(\xi)=0\qquad
\forall \xi\in 
{\mathbb{R}}^{n}\setminus\bigcup_{z\in {\mathbb{Z}}^{n}}\left(qz+{\mathbb{B}}_{n}(0,\delta/4)\right)\,,
\]
and accordingly
\[
\left|\frac{\partial}{\partial\xi_{j}}g_{\delta}(\xi)\right|
\leq \frac{1}{\delta}\sup|g'|\qquad\forall \xi\in {\mathbb{R}}^{n}\,,
\]
for all $\delta\in]0,1[$.  Since the functions $h$ and $\frac{\partial h}{\partial x_{j}}$ are $q$-periodic and continuous  in ${\mathbb{R}}^{n}\setminus
q{\mathbb{Z}}^{n}$, the same functions are  bounded on 
${\mathbb{R}}^{n}\setminus\bigcup_{z\in {\mathbb{Z}}^{n}}
\left(qz+{\mathbb{B}}_{n}(0,\delta/16)\right)$. Since $g_{\delta}$ vanishes  on $\bigcup_{z\in {\mathbb{Z}}^{n}}\left(qz+{\mathbb{B}}_{n}(0,\delta/8)\right)$, we have 
\begin{eqnarray}
\nonumber
\lefteqn{
\left|
\frac{\partial}{\partial x_{j}}
\left\{
g_{\delta}(x-y)h(x-y)\varphi(y)
\right\}
\right|
}
\\ \nonumber
&&\qquad
\leq \frac{1}{\delta}\sup|g'|
\left( \sup_{
{\mathbb{R}}^{n}\setminus\bigcup_{z\in {\mathbb{Z}}^{n}}
\left(qz+{\mathbb{B}}_{n}(0,\delta/16)\right)
}|h|\right)  |\varphi(y)|
\\ \nonumber
&&\qquad\qquad 
+
\max_{j\in\{1,{...},n\}}
\left( \sup_{
{\mathbb{R}}^{n}\setminus\bigcup_{z\in {\mathbb{Z}}^{n}}
\left(qz+{\mathbb{B}}_{n}(0,\delta/16)\right)
}\left|
\frac{\partial h}{\partial x_{j}}
\right|\right)\|g\|_{
L^{\infty}({\mathbb{R}} )
}  |\varphi(y)| \,,
\end{eqnarray}
for all $y\in\Omega$ and $x\in {\mathbb{R}}^{n}$. Since $\varphi\in L^{1}(\Omega)$, the differentiability theorem for integrals depending on a parameter implies that
\[
\frac{\partial u_{\delta}^{+}}{\partial x_{j}}(x)= \int_{\Omega}
\frac{\partial  }{\partial x_{j}}\left[
g_{\delta}(x-y)h(x-y)
\right]\varphi (y)\,dy\qquad\forall x\in {\mathbb{R}}^{n}\,.
\]
Hence, $u_{\delta}^{+}\in C^{1}_{q}(
{\mathbb{R}}^{n}
)$. In order to prove that ${\mathcal{P}}^{+}[h,\varphi]
\in C^{1}_{q}({\mathbb{R}}^{n}
)$, it suffices to show that
\begin{eqnarray}
\label{q1poi6}
&\lim_{\delta\to 0}u_{\delta}^{+}={\mathcal{P}}^{+}[h,\varphi]
&{\mathrm{uniformly\ in}}\  {\mathbb{R}}^{n}\,,
\\
\label{q1poi7}
&\lim_{\delta\to 0}\frac{\partial u_{\delta}^{+}}{\partial x_{j}}={\mathcal{P}}^{+}[\frac{\partial h}{\partial x_{j}},\varphi]
&{\mathrm{uniformly\ in}}\  {\mathbb{R}}^{n}\,,
\end{eqnarray}
for all $j\in \{1,{...},n\}$. We first prove (\ref{q1poi6}). Since $1-g_{\delta}(x-y)=0$ when $x-y\in {\mathbb{R}}^{n}\setminus\bigcup_{z\in {\mathbb{Z}}^{n}}
\left(qz+{\mathbb{B}}_{n}(0,\delta/4)\right)$ and ${\mathbb{B}}_{n}(0,\delta/4)\subseteq\tilde{Q}$ for all $\delta\in]0,\min_{j=1,{...},n}q_{jj}[$
and ${\mathbb{Z}}_{Q}$ has $2^{n}$ elements, we have 
\begin{eqnarray*}
\lefteqn{
| {\mathcal{P}}^{+}[h,\varphi] (x)-u_{\delta}^{+}(x)|
}
\\ \nonumber
&&\qquad
=\left|\int_{
\left(
\bigcup_{z\in {\mathbb{Z}}^n}
\left(qz+{\mathbb{B}}_{n}(x,\delta/4)\right)
\right)\cap\Omega
} (1-g_{\delta}(|x-y|)) h(x-y)\varphi (y)\,dy\right|
\\ \nonumber
&&\qquad\leq
(1+\|g\|_{L^\infty({\mathbb{R}})})
\|\varphi\|_{L^{\infty}(\Omega)}
\int_{
\bigcup_{z\in {\mathbb{Z}}_{Q}}
\left(qz+{\mathbb{B}}_{n}(0,\delta/4)\right)
}|h( \eta)|\,d\eta
\\ \nonumber
&&\qquad\leq
(1+\|g\|_{L^\infty({\mathbb{R}})})
\|\varphi\|_{L^{\infty}(\Omega)}
2^{n}\|h\|_{A^{0}_{q,\lambda} }
\int_{\delta\tilde{Q} }|y|^{-\lambda}\,dy\qquad\forall x\in {\mathbb{R}}^{n}\,,
\end{eqnarray*}
for all $\delta\in]0,\min\{
1,
\min_{j=1,{...},n}q_{jj}
\}[$. Since $|y|^{-\lambda}$ is integrable in 
$\tilde{Q}$, $\lim_{\delta\to 0}\int_{ \delta\tilde{Q} }
|y|^{-\lambda}\,dy=0$ and the limiting relation (\ref{q1poi6}) follows. 

We now consider the limiting relation  (\ref{q1poi7}). Since $g'$ vanishes outside of the interval $[1/8,1/4]$, the same argument  we have exploited above shows that
\begin{eqnarray*}
\lefteqn{
\left|
{\mathcal{P}}^{+}[\frac{\partial h}{\partial x_{j}},\varphi](x)-
\frac{\partial u_{\delta}^{+} }{  \partial x_{j}  }(x)
\right|
}
\\ \nonumber
&&\qquad
\leq 
\left|
\int_{\Omega} (1-g_{\delta}( x-y ))
\frac{\partial h}{\partial x_{j} }(x-y)\varphi (y)\,dy 
\right|
\\ \nonumber
&&\qquad\qquad+
\left|
 \int_{ \Omega}
 \frac{\partial  }{\partial x_{j} }
 \left\{
 g_{\delta}(x-y)
 \right\}h(x-y)\varphi (y)\,dy 
 \right|
 \\ \nonumber
&&\qquad
\leq 
(1+\|g\|_{L^\infty({\mathbb{R}})})
\|\varphi\|_{L^{\infty}(\Omega)}
2^{n}\left\|\frac{\partial h}{\partial x_{j} }\right\|_{A^{0}_{q,\lambda+1} }
\int_{\delta\tilde{Q} }|y|^{-\lambda-1}\,dy
\\ \nonumber
&&\qquad\qquad+
\frac{1}{\delta}\sup|g'|
\|\varphi\|_{L^{\infty}(\Omega)}
2^{n}\left\| h \right\|_{A^{0}_{q,\lambda} }
\int_{\delta\tilde{Q} }|y|^{-\lambda}\,dy
\end{eqnarray*}
for all $\delta\in]0,\min\{
1,
\min_{j=1,{...},n}q_{jj}
\}[$ and $j\in\{1,{...},n\}$. Next we note that
\[
 \int_{\delta\tilde{Q} }|y|^{-\lambda-1}\,dy
=
 \delta^{n-\lambda-1}
\int_{ \tilde{Q} }|y|^{-\lambda-1}\,dy
\,,
\quad
\frac{1}{\delta}\int_{\delta\tilde{Q} }|y|^{-\lambda}\,dy
=
\frac{1}{\delta}\delta^{n-\lambda}
\int_{ \tilde{Q} }|y|^{-\lambda}\,dy 
\,.
\]
Since $n-\lambda-1>0$, we conclude that the limiting relation (\ref{q1poi7}) holds. Hence, we have ${\mathcal{P}}^{+}[h,\varphi] \in 
C^{1}_{q}(
{\mathbb{R}}^{n}
)$. 

The proof of statement (iv) follows the lines of that of statement (ii) by replacing the integration in $\Omega$ with that on $Q\setminus{\mathrm{cl}}\Omega$, and it is accordingly omitted. \end{proof} 

\vspace{\baselineskip}  
 
Next we prove the following lemma. 
 \begin{lemma}
\label{qpoder}
Let $\lambda\in ]0,n-1[$. Let $\Omega$ be a bounded open Lipschitz subset of ${\mathbb{R}}^{n}$ such that 
${\mathrm{cl}}\Omega\subseteq Q$. Then the following statements hold. 
\begin{enumerate}
\item[(i)] Let $\Omega\neq \emptyset$. If $(h,\varphi)\in  A^{1}_{q,\lambda} \times C^{1}({\mathrm{cl}} \Omega)$, then 
\begin{equation}
\label{qpoder1}
\frac{\partial  }{\partial x_{j}}{\mathcal{P}}^{+}[h,\varphi](x)
={\mathcal{P}}^{+}[h,\frac{\partial \varphi }{\partial x_{j}}](x)
-\int_{\partial\Omega}h(x-y)\varphi(y)(\nu_{\Omega})_{j}(y)\,d\sigma_{y}
 \,,
\end{equation}
for all $x\in {\mathbb{R}}^{n}$.
\item[(ii)]  Let $\Omega\neq \emptyset$. If $(h,\varphi)\in  A^{1}_{q,\lambda} \times C^{1}({\mathrm{cl}}Q\setminus  \Omega)$, then 
\begin{eqnarray}
\label{qpoder2}
\lefteqn{
\frac{\partial  }{\partial x_{j}}{\mathcal{P}}^{-}[h,\varphi](x)
={\mathcal{P}}^{-}[h,\frac{\partial \varphi }{\partial x_{j}}](x)
}
\\
\nonumber
&&
+\int_{\partial\Omega}h(x-y)\varphi(y)(\nu_{\Omega})_{j}(y)\,d\sigma_{y}
-
\int_{\partial Q}h(x-y)\varphi(y)(\nu_{Q})_{j}(y)\,d\sigma_{y}
 \,,
\end{eqnarray}
for all $x\in {\mathbb{R}}^{n}$.
\item[(iii)] Let $\Omega = \emptyset$. If $(h,\varphi)\in  A^{1}_{q,\lambda} \times C^{1}({\mathrm{cl}}Q )$, then 
\begin{equation}
\label{qpoder2a}
\frac{\partial  }{\partial x_{j}}{\mathcal{P}}^{-}[h,\varphi](x)
={\mathcal{P}}^{-}[h,\frac{\partial \varphi }{\partial x_{j}}](x)
-
\int_{\partial Q}h(x-y)\varphi(y)(\nu_{Q})_{j}(y)\,d\sigma_{y}
 \,,
\end{equation}
for all $x\in {\mathbb{R}}^{n}$.
\end{enumerate}
\end{lemma}
\begin{proof} We first consider statement (i). Again, we proceed by a classical argument. By the previous proposition, we have 
\begin{eqnarray}
\label{qpoder3}
\lefteqn{
\frac{\partial  }{\partial x_{j}}{\mathcal{P}}^{+}[h,\varphi](x)=
\int_{\Omega}\frac{\partial h }{\partial x_{j}}(x-y)\varphi (y)\,dy
}
\\ \nonumber
&&\qquad
=-\int_{\Omega}\frac{\partial  }{\partial y_{j}}(h(x-y))\varphi(y)\,dy
\\ \nonumber
&&\qquad
=-\int_{\Omega}\frac{\partial  }{\partial y_{j}}\left(h(x-y)\varphi(y)
\right)\,dy
+
\int_{\Omega}h(x-y)\frac{\partial \varphi }{\partial y_{j}}(y)\,dy\qquad\forall x\in {\mathbb{R}}^{n}\,.
\end{eqnarray}
Since the singularities of the kernel $h(x-y)$ are weak and $h$ is continuous in ${\mathbb{R}}^{n}\setminus q{\mathbb{Z}}^{n}$, and $\varphi (\cdot)(\nu_{\Omega})_{j}(\cdot)$ is integrable on $\partial\Omega$, the Vitali Convergence Theorem implies that the integral $\int_{\partial\Omega}h(x-y)\varphi (y)(\nu_{\Omega})_{j}(y)
\,d\sigma_{y}$ is continuous and $q$-periodic in $x\in {\mathbb{R}}^{n}$. Then Propositions \ref{qopoi} (ii) and
 \ref{q1poi} (ii) imply that both the left and the right hand sides of 
 (\ref{qpoder1}) are continuous and $q$-periodic in $x\in {\mathbb{R}}^{n}$. Hence, it suffices to verify (\ref{qpoder1})  on 
 $\Omega\cup (Q\setminus{\mathrm{cl}}\Omega)$. We first consider case $x\in \Omega$. Thus we now fix
$x\in \Omega$, and we take 
$\epsilon_{x}\in]0,{\mathrm{dist}}\,(x,\partial\Omega)[$. Then 
${\mathrm{cl}}{\mathbb{B}}_{n}(x,\epsilon)\subseteq\Omega$ and the set 
\[
\Omega_{\epsilon}\equiv\Omega\setminus
{\mathrm{cl}}{\mathbb{B}}_{n}(x,\epsilon)
\]
is an open Lipschitz set for all $\epsilon\in]0,\epsilon_{x}[$. 
By the Divergence Theorem, we have 
\begin{eqnarray}
\label{qpoder4}
\lefteqn{
\int_{\Omega}\frac{\partial  }{\partial y_{j}}\left(h(x-y)\varphi(y)
\right)\,dy
}
\\ \nonumber
&&\quad
=\int_{\Omega_{\epsilon}}\frac{\partial  }{\partial y_{j}}\left(h(x-y)\varphi(y)
\right)\,dy
+
\int_{{\mathbb{B}}_{n}(x,\epsilon)}\frac{\partial  }{\partial y_{j}}\left(h(x-y)\varphi(y)
\right)\,dy
\\ \nonumber
&&\quad
=\int_{\partial\Omega}h(x-y)\varphi(y)(\nu_{\Omega})_{j}(y)\,d\sigma_{y}
+\int_{\partial {\mathbb{B}}_{n}(x,\epsilon) }h(x-y)\varphi(y)
\frac{x_{j}-y_{j}}{|x-y|}\,d\sigma_{y}
\\ \nonumber
&&\qquad
-\int_{{\mathbb{B}}_{n}(x,\epsilon)}\frac{\partial h }{\partial x_{j}}  (x-y)\varphi(y)\,dy
+\int_{{\mathbb{B}}_{n}(x,\epsilon)} h(x-y)\frac{\partial \varphi  }{\partial y_{j}} (y)\,dy\,,
\end{eqnarray}
for all $\epsilon\in]0,\epsilon_{x}[$.
By the inclusion ${\mathrm{cl}}{\mathbb{B}}_{n}(x,\epsilon_{x})
\subseteq\Omega\subseteq Q$, we have $
{\mathrm{cl}}{\mathbb{B}}_{n}(0,\epsilon_{x})
\subseteq\tilde{Q}$ and
\begin{eqnarray}
\label{qpoder5}
\lefteqn{
\left|
\int_{\partial {\mathbb{B}}_{n}(x,\epsilon) }h(x-y)\varphi(y)
\frac{x_{j}-y_{j}}{|x-y|}\,d\sigma_{y}
\right|
}
\\ \nonumber
&&\qquad
\leq
\left(
\sup_{x\in {\mathrm{cl}}{\mathbb{B}}_{n}(0,
 \epsilon_{x} 
) \setminus\{0\}
}|h(x)|\,|x|^{\lambda}
\right)\|\varphi\|_{L^{\infty}(\Omega)}
\int_{\partial {\mathbb{B}}_{n}(0,\epsilon)
}|y|^{-\lambda}\,d\sigma_y
\\ \nonumber
&&\qquad
\leq
\left(
\sup_{x\in \tilde{Q} \setminus\{0\}
}|h(x)|\,|x|^{\lambda}
\right)\|\varphi\|_{L^{\infty}(\Omega)}
s_{n}\epsilon^{n-1-\lambda}\,,
\end{eqnarray}
for all $\epsilon\in]0,\epsilon_{x}[$. By  
Proposition \ref{q1poi} (i)  the functions $\frac{\partial h}{\partial x_{j}}(x-y)\varphi (y)$, $h(x-y)\frac{\partial \varphi}{\partial x_{j}}(y)$ are integrable 
 in the variable $y\in \Omega$. Then we have
\begin{eqnarray}
\label{qpoder6}
\lim_{\epsilon\to 0}
\int_{{\mathbb{B}}_{n}(x,\epsilon)}\frac{\partial h }{\partial x_{j}}  (x-y)\varphi(y)\,dy=0\,,
\\
\nonumber
\lim_{\epsilon\to 0}
\int_{{\mathbb{B}}_{n}(x,\epsilon)} h(x-y)\frac{\partial \varphi  }{\partial y_{j}} (y)\,dy=0\,.
\end{eqnarray}
By (\ref{qpoder5}) and (\ref{qpoder6}), we can take the limit as $\epsilon$ tends to $0$ in (\ref{qpoder4}) and deduce that 
\begin{equation}
\label{qpoder6a}
\int_{\Omega}\frac{\partial  }{\partial y_{j}}\left(h(x-y)\varphi(y)
\right)\,dy
=
\int_{\partial\Omega}h(x-y)\varphi(y)(\nu_{\Omega})_{j}(y)\,d\sigma_{y}\,.
\end{equation}
Then equality (\ref{qpoder3}) implies that formula (\ref{qpoder1}) holds.

We now turn to consider case $x\in Q\setminus{\mathrm{cl}}\Omega$. Since $h(x-y)\varphi (y)$ is continuously differentiable in $y\in {\mathrm{cl}}\Omega$, the Divergence Theorem implies that (\ref{qpoder6a}) holds. Hence, equality (\ref{qpoder3}) implies that equality (\ref{qpoder1}) holds.

We now prove statement (ii). By the previous proposition, we have 
\begin{eqnarray}
\label{qpoder7}
\lefteqn{
\frac{\partial  }{\partial x_{j}}{\mathcal{P}}^{-}[h,\varphi](x)=
\int_{Q\setminus{\mathrm{cl}}\Omega}\frac{\partial h }{\partial x_{j}}(x-y)\varphi (y)\,dy
}
\\ \nonumber
&&\qquad
=-\int_{Q\setminus{\mathrm{cl}}\Omega}\frac{\partial  }{\partial y_{j}}(h(x-y))\varphi(y)\,dy
\\ \nonumber
&&\qquad
=-\int_{Q\setminus{\mathrm{cl}}\Omega}\frac{\partial  }{\partial y_{j}}\left(h(x-y)\varphi(y)
\right)\,dy
+
\int_{Q\setminus{\mathrm{cl}}\Omega}h(x-y)\frac{\partial \varphi }{\partial y_{j}}(y)\,dy\,,
\end{eqnarray}
for all $x\in {\mathbb{R}}^{n}$. Since the singularities of the kernel $h(x-y)$ are weak and $h$ is continuous in ${\mathbb{R}}^{n}\setminus q{\mathbb{Z}}^{n}$, and $\varphi (\cdot)(\nu_{\Omega})_{j}(\cdot)$, $\varphi (\cdot)(\nu_{Q})_{j}(\cdot)$
are  integrable on $\partial\Omega$ and on $\partial Q$, respectively, 
the Vitali Convergence Theorem implies that the integrals on $\partial\Omega$ and on $\partial Q$ in the right hand side of (\ref{qpoder2}) are continuous and $q$-periodic in $x\in {\mathbb{R}}^{n}$. Then Propositions \ref{qopoi} (v) and \ref{q1poi}  (iv) imply that both the left  and the right hand sides of (\ref{qpoder2}) are continuous in $x\in {\mathbb{R}}^{n}$. Hence, it suffices to verify that 
equality  (\ref{qpoder2}) holds for $x\in \Omega\cup (Q\setminus{\mathrm{cl}}\Omega)$. We first consider case $x\in \Omega$. Since $h(x-y)\varphi (y)$ is continuously differentiable in $y\in {\mathrm{cl}}Q\setminus\Omega$, the Divergence Theorem implies that
\begin{eqnarray}
\label{qpoder7a}
\lefteqn{
\int_{ Q\setminus{\mathrm{cl}}\Omega }\frac{\partial  }{\partial y_{j}}\left(h(x-y)\varphi(y)
\right)\,dy}
\\ \nonumber
&&\qquad
=
\int_{\partial Q}h(x-y)\varphi (y)(\nu_{Q})_{j}(y)\,d\sigma_{y}
-\int_{\partial\Omega}h(x-y)\varphi(y)(\nu_{\Omega})_{j}(y)\,d\sigma_{y}\,.
\end{eqnarray}
Hence, equality (\ref{qpoder7}) implies that equality
(\ref{qpoder2}) holds.  Next we consider the case when $x\in Q\setminus{\mathrm{cl}}\Omega$. Let $\epsilon_{x}\in]0,{\mathrm{dist}}\,(x,\partial(Q\setminus{\mathrm{cl}}\Omega) )[$. Then 
${\mathrm{cl}}{\mathbb{B}}_{n}(x,\epsilon)\subseteq
Q\setminus{\mathrm{cl}}\Omega$ and the set 
\[
V_{\epsilon}\equiv
(Q\setminus{\mathrm{cl}}\Omega)
\setminus
{\mathrm{cl}}{\mathbb{B}}_{n}(x,\epsilon)
\]
is an open Lipschitz set  for all $\epsilon\in]0,\epsilon_{x}[$. 
By the Divergence Theorem, we have 
\begin{eqnarray}
\label{qpoder8}
\lefteqn{
\int_{ Q\setminus{\mathrm{cl}}\Omega }\frac{\partial  }{\partial y_{j}}\left(h(x-y)\varphi(y)
\right)\,dy
}
\\ \nonumber
&&\quad
=\int_{V_{\epsilon}}\frac{\partial  }{\partial y_{j}}\left(h(x-y)\varphi(y)
\right)\,dy
+
\int_{{\mathbb{B}}_{n}(x,\epsilon)}\frac{\partial  }{\partial y_{j}}\left(h(x-y)\varphi(y)
\right)\,dy
\\ \nonumber
&&\quad
=\int_{\partial Q}h(x-y)\varphi(y)(\nu_{Q})_{j}(y)\,d\sigma_{y}
-\int_{\partial\Omega}h(x-y)\varphi(y)(\nu_{\Omega})_{j}(y)\,d\sigma_{y}
\\ \nonumber
&&\quad\quad
+\int_{\partial {\mathbb{B}}_{n}(x,\epsilon) }h(x-y)\varphi(y)
\frac{x_{j}-y_{j}}{|x-y|}\,d\sigma_{y}
\\ \nonumber
&&\qquad
-\int_{{\mathbb{B}}_{n}(x,\epsilon)}\frac{\partial h }{\partial x_{j}}  (x-y)\varphi(y)\,dy
+\int_{{\mathbb{B}}_{n}(x,\epsilon)} h(x-y)\frac{\partial \varphi  }{\partial y_{j}} (y)\,dy
\end{eqnarray}
for all $\epsilon\in]0,\epsilon_{x}[$. Then by arguing precisely as in the proof of statement (i), we can prove that
\begin{eqnarray*}
\lim_{\epsilon\to 0}
\int_{\partial{\mathbb{B}}_{n}(x,\epsilon)}
h(x-y)\varphi(y)\frac{x_{j}-y_{j}}{|x-y|}\,dy=0\,,
\\
\lim_{\epsilon\to 0}
\int_{{\mathbb{B}}_{n}(x,\epsilon)}\frac{\partial h }{\partial x_{j}}  (x-y)\varphi(y)\,dy=0\,,
\\
\nonumber
\lim_{\epsilon\to 0}
\int_{{\mathbb{B}}_{n}(x,\epsilon)} h(x-y)\frac{\partial \varphi  }{\partial y_{j}} (y)\,dy=0\,.
\end{eqnarray*}
Then by taking the limit in (\ref{qpoder8}) as $\epsilon$ tends to  $0$, we deduce that equality (\ref{qpoder7a}) holds. Then equality  (\ref{qpoder7}) implies that formula 
(\ref{qpoder2}) holds. The proof of statement (iii) can be effected by a simplification of the proof of statement (ii) and it  is accordingly omitted.  \end{proof} 

\vspace{\baselineskip}

\noindent
{\bf Remark.}   Under the same assumptions of Lemma \ref{qpoder} (ii), (iii),  if $\varphi$  is the restriction of an element of 
$C^{1}_{q}({\mathrm{cl}}{\mathbb{S}}[\Omega]^{-})$, then we have 
\begin{equation}\label{qpoderz}
\int_{\partial Q}h(x-y)\varphi (y)(\nu_{Q})_{j}(y)\,d\sigma_{y}=0\,,
\end{equation}
for all $j\in\{1,{...},n\}$, and the second integral in the right hand side of (\ref{qpoder2}) and the integral in the right hand side of  (\ref{qpoder2a})  vanish. 

\vspace{\baselineskip}

Next we introduce a class of weakly singular $q$-periodic kernels, which we consider in our main results. 

\begin{definition}
\label{qh1r}
 Let  $G$ be an open subset of ${\mathbb{R}}^{n}$
 such that ${\mathrm{cl}}G$ is bounded and contained in ${\mathbb{R}}^{n}\setminus q{\mathbb{Z}}^{n}$. 
  Let $\rho\in]0,+\infty[$. Then we set
\[
H_{q}^{\lambda,\rho}(G)
\equiv
\left\{
h\in A^{1}_{q,\lambda} :\,
h_{|{\mathrm{cl}} G
}\in 
C^{0}_{\omega,\rho}({\mathrm{cl}} G)
\right\}\,,
\]
and we set
\[
\|h\|_{H_{q}^{\lambda,\rho}(G)}
\equiv
\|h\|_{A^{1}_{q,\lambda} }+
\|h\|_{
C^{0}_{\omega,\rho}({\mathrm{cl}} G)
}
\qquad\forall h\in
H_{q}^{\lambda,\rho}(G)\,.
\]
  Here we understand that if $G=\emptyset$, then $H_{q}^{\lambda,\rho}(G)=A^{1}_{q,\lambda}$ and $\|\cdot\|_{H_{q}^{\lambda,\rho}(G)}=
\|\cdot\|_{A^{1}_{q,\lambda} }$.
\end{definition}
Since both $A^{1}_{q,\lambda}$ and $C^{0}_{\omega,\rho}({\mathrm{cl}} G)$ are Banach spaces, also  $(H_{q}^{\lambda,\rho}(G),\|\cdot\|_{
H_{q}^{\lambda,\rho}(G)
})$ is a Banach space.

We are now ready to prove here below  in Theorems \ref{qropo} and \ref{qropoo} our main results on the continuity of the bilinear maps which take the pair consisting of a kernel and of a  density function to the volume potentials. 
In  Theorem \ref{qropo} we consider the restriction of the volume potential to a domain $\mathbb{S}[\Omega_1]$, with $\mathrm{cl}\Omega_1\subseteq \Omega\subseteq \mathrm{cl}\Omega\subseteq Q$, while in  Theorems \ref{qropoo} we study the restriction to a domain $\mathbb{S}[\Omega_2]^-$, with  ${\mathrm{cl}}\Omega
\subseteq \Omega_{2}\subseteq {\mathrm{cl}}\Omega_{2}\subseteq Q$.

\begin{theorem}
\label{qropo}
Let $\rho\in]0,+\infty[$, $\lambda\in ]0,n-1[$. 
Let $\Omega$ be a bounded open Lipschitz subset of ${\mathbb{R}}^{n}$
such that ${\mathrm{cl}}\Omega\subseteq Q$. Let $\Omega_{1}$ be a nonempty open subset of ${\mathbb{R}}^{n}$ such that ${\mathrm{cl}}\Omega_{1}\subseteq\Omega$. Let
$G$ be an open relatively compact neighborhood of the set
\[
\left\{
t-s:\, t\in {\mathrm{cl}}\Omega_{1}, s\in\partial\Omega
\right\}\,,
\]
such that $({\mathrm{cl}}G)\cap q{\mathbb{Z}}^{n}=\emptyset$. 
Then the following statements hold.
\begin{enumerate}
\item[(i)] The restriction of ${\mathcal{P}}^{+}[h,\varphi]$ to 
${\mathrm{cl}}{\mathbb{S}}[\Omega_{1}]$ belongs to  $C^{0}_{q,\omega,\rho}
({\mathrm{cl}}{\mathbb{S}}[\Omega_{1}])$ for all  pairs $(h,\varphi)\in H_{q}^{\lambda,\rho}(G)\times C^{0}_{\omega,\rho}({\mathrm{cl}}\Omega)$ and   
 the map from $H_{q}^{\lambda,\rho}(G)\times C^{0}_{\omega,\rho}({\mathrm{cl}}\Omega)$ to $C^{0}_{q,\omega,\rho}
({\mathrm{cl}}{\mathbb{S}}[\Omega_{1}])$ which takes $(h,\varphi)$ to $
{\mathcal{P}}^{+}[h,\varphi]_{|  {\mathrm{cl}}{\mathbb{S}}[\Omega_{1}]  }$
is bilinear and continuous. 
\item[(ii)] The restriction of ${\mathcal{P}}^{-}[h,\varphi_{|Q\setminus{\mathrm{cl}}\Omega}]$ to 
${\mathrm{cl}}{\mathbb{S}}[\Omega_{1}]$ belongs to  $C^{0}_{q,\omega,\rho}
({\mathrm{cl}}{\mathbb{S}}[\Omega_{1}])$ for all pairs  $(h,\varphi)\in H_{q}^{\lambda,\rho}(G)\times C^{0}_{q,\omega,\rho}
({\mathrm{cl}}{\mathbb{S}}[\Omega ]^{-})$ and 
 the map from $H_{q}^{\lambda,\rho}(G)\times C^{0}_{q,\omega,\rho}
({\mathrm{cl}}{\mathbb{S}}[\Omega ]^{-})$ to $C^{0}_{q,\omega,\rho}
({\mathrm{cl}}{\mathbb{S}}[\Omega_{1}])$ which takes $(h,\varphi)$ to $
{\mathcal{P}}^{-}[h,\varphi_{|Q\setminus{\mathrm{cl}}\Omega}]_{|  {\mathrm{cl}}{\mathbb{S}}[\Omega_{1}]  }$
is bilinear and continuous. 
\end{enumerate}
\end{theorem}
\begin{proof} We first consider statement (i), and we prove that if $m\in {\mathbb{N}}\setminus\{0\}$ and if $(h,\varphi)\in
H^{\lambda,\rho}_{q}(G)\times C^{m} ({\mathrm{cl}}\Omega)$, then 
$
{\mathcal{P}}^{+}[h,\varphi]_{|  {\mathrm{cl}}
{\mathbb{S}}[\Omega_{1}]
 }
\in C^{m}_{q}
({\mathrm{cl}}
{\mathbb{S}}[\Omega_{1}])$ and
\begin{eqnarray}
	\label{qropo2}
	\lefteqn{
\partial^{\beta}
{\mathcal{P}}^{+}[h,\varphi](x)=
{\mathcal{P}}^{+}[h,\partial^{\beta}\varphi](x)
}
\\
&&
\nonumber\qquad\qquad\qquad\qquad
-\sum_{k=1}^{n}\sum_{l_{k}=0}^{\beta_{k}-1}
\partial^{\beta_{n}}_{x_{n}}{...}\partial^{\beta_{k+1}}_{x_{k+1}}\partial^{l_{k}}_{x_{k}}
\biggl\{\biggr.
\int_{\partial\Omega}h(x-y)
\\
&&
\nonumber\qquad\qquad\qquad\qquad\qquad
\times
\left(\partial^{\beta_{k}-1-l_{k}}_{y_{k}}
\partial^{\beta_{k-1}}_{y_{k-1}}{...}\partial^{\beta_{1}}_{y_{1}}\varphi(y)\right)
(\nu_{\Omega})_{k}(y)\,d\sigma_{y}
\biggl.\biggr\}\,,
\end{eqnarray}
for all $x\in{\mathrm{cl}}{\mathbb{S}}[\Omega_{1}]$ and for all $\beta\in {\mathbb{N}}^{n}$ such that $|\beta|\leq m$,  where we understand that 
$\sum_{l_{k}=0}^{\beta_{k}-1}$ is omitted 
if $\beta_{k}=0$. 
Since $h$ is $q$-periodic, it suffices to prove (\ref{qropo2}) for  $x\in{\mathrm{cl}} \Omega_{1} $. If $m=1$, then the statement follows by Lemma \ref{qpoder} (i). 

Next we  assume that the statement holds for $m$  and we prove it for $m+1$. 
Let $(h,\varphi)\in
H^{\lambda,\rho}_{q}(G)\times C^{m+1} ({\mathrm{cl}}\Omega)$. By the inductive assumption, ${{\mathcal{P}}^{+}[h,\frac{\partial \varphi }{\partial x_{j}}]_{|{\mathrm{cl}}{\mathbb{S}}[\Omega_{1}]}}\in C^{m }_{q} ({\mathrm{cl}}{\mathbb{S}}[\Omega_{1}])$ for all $j\in\{1,{...},n\}$.

 Since 
$h\in C^{m}({\mathrm{cl}}G
 )$ and $\varphi\,,(\nu_{\Omega})_{j}\in C^{0}(\partial\Omega)$, the classical differentiability theorem  for integrals depending on a parameter implies that the second term in the right hand side of formula (\ref{qpoder1}) defines a function of class $C^{m}_{q}(
 {\mathrm{cl}}{\mathbb{S}}[\Omega_{1}])$. Then formula (\ref{qpoder1}) implies that $\frac{\partial}{\partial x_{j}}{\mathcal{P}}^{+}[h,\varphi]$ belongs to $C^{m}_{q}({\mathrm{cl}}{\mathbb{S}}[\Omega_{1}])$.
Hence, ${\mathcal{P}}^{+}[h,\varphi]_{|
{\mathrm{cl}\mathbb{S}}[\Omega_{1}]
}\in  C^{m+1}_{q}(
 {\mathrm{cl}}{\mathbb{S}}[\Omega_{1}])$. 

Next one proves the formula for the derivatives  by following the lines  of the corresponding argument  of \cite[p.~856]{La05}: first one proves the formula for $\partial^{\beta}=\partial^{\beta_{j}}_{x_{j}}$ by finite induction on $\beta_{j}$, then one proves the formula for 
 $\partial^{\beta}=\partial^{\beta_{1}}_{x_{1}}{...}
 \partial^{\beta_{j}}_{x_{j}}$ by finite induction on $j\in\{1,{...},n\}$, and  finally one deduces that the formula holds for $|\beta|\leq m+1$.

If $(h,\varphi)\in
H^{\lambda,\rho}(G)\times C^{\infty} ({\mathrm{cl}}\Omega)$, then by applying the above  statement for all $m\in {\mathbb{N}}\setminus\{0\}$ we deduce that
$
{\mathcal{P}}^{+}[h,\varphi]_{|{\mathrm{cl}}{\mathbb{S}}[\Omega_{1}]}
$ belongs to $ C^{\infty}_{q} 
({\mathrm{cl}}{\mathbb{S}}[\Omega_{1}])$ and that formula (\ref{qropo2}) holds for all order derivatives.

We now assume that $(h,\varphi)\in
H^{\lambda,\rho}_{q}(G)\times C^{0}_{\omega,\rho}({\mathrm{cl}}\Omega)$ and we turn to estimate the sup-norm in ${\mathrm{cl}}\Omega_{1}$ of the sum in the right hand side of (\ref{qropo2}), which we denote by  $I$.

We abbreviate by $I(k,l_{k})$ the $(k,l_{k})$-th term in the sum $I$, and we now estimate the supremum of $I(k,l_{k})$ in ${\mathrm{cl}}\Omega_{1}$. We can clearly assume that $\beta_{k}>0$. Then we have 
\begin{eqnarray}
	\label{qropo3}
	\lefteqn{
\sup_{ {\mathrm{cl}}\Omega_{1}}|I(k,l_{k})|  
}
\\
&&
\nonumber =
\sup_{x\in {\mathrm{cl}}\Omega_{1}}
\biggl|
\partial^{\beta_{n}}_{x_{n}}{...}\partial^{\beta_{k+1}}_{x_{k+1}}\partial^{l_{k}}_{x_{k}}
\left\{
\int_{\partial\Omega}h(x-y)
\partial^{\beta_{k}-1-l_{k}}_{y_{k}}\partial^{\beta_{k-1}}_{y_{k-1}}
{...}\partial^{\beta_{1}}_{y_{1}}\varphi (y)
(\nu_{\Omega})_k(y)\,d\sigma_{y}
\right\}
\biggr|
\\
&&
\nonumber
\qquad 
\leq\int_{\partial\Omega}
\sup_{\xi\in G}
\left|
\partial^{\beta_{n}}_{\xi_{n}}{...} \partial^{\beta_{k+1}}_{\xi_{k+1}}
 \partial^{l_{k}}_{\xi_{k}}h(\xi)
\right|
\,
\left|
\partial^{\beta_{k}-1-l_{k}}_{y_{k}}\partial^{\beta_{k-1}}_{y_{k-1}}{...}\partial^{\beta_{1}}_{y_{1}}\varphi (y)
\right|\,d\sigma_{y}\,.
\end{eqnarray}
Since $h\in H^{\lambda,\rho}_{q}(G)$, we have 
\[
\sup_{\xi\in G}
\left|
\partial^{\beta_{n}}_{\xi_{n}}{...} \partial^{\beta_{k+1}}_{\xi_{k+1}}
 \partial^{l_{k}}_{\xi_{k}}h(\xi)
\right|
\leq
\|h\|_{  H^{\lambda,\rho}_{q}(G)  }
\frac{(\beta_{n}+{...}+\beta_{k+1}+l_{k})!}{
\rho^{\beta_{n}+{...}+\beta_{k+1}+l_{k}}
}\,.
\]
 Moreover, 
\[
\left|
\partial^{\beta_{k}-1-l_{k}}_{y_{k}}\partial^{\beta_{k-1}}_{y_{k-1}}{...}\partial^{\beta_{1}}_{y_{1}}\varphi (y)
\right|
\leq
\|\varphi\|_{ C^{0}_{\omega,\rho}({\mathrm{cl}}\Omega)}
\frac{(\beta_{1}+{...}+\beta_{k-1}+\beta_{k}-1-l_{k})!}{
\rho^{
\beta_{1}+{...}+\beta_{k-1}+\beta_{k}-1-l_{k}
}}\,,
\]
for all $y \in \mathrm{cl}\Omega$. Then we have 
\begin{eqnarray}
\nonumber
\lefteqn{
\sup_{ {\mathrm{cl}}\Omega_{1}}|I(k,l_{k})| 
\leq 
m_{n-1}(\partial\Omega)
\|h\|_{  H^{\lambda,\rho}_{q}(G)  }
\|\varphi\|_{ C^{0}_{\omega,\rho}({\mathrm{cl}}\Omega)}
\frac{1}{  \rho^{|\beta|-1}
}}
\\ \nonumber
&&\qquad\qquad
\times
(\beta_{n}+{...}+\beta_{k+1}+l_{k})! 
(\beta_{1}+{...}+\beta_{k-1}+\beta_{k}-1-l_{k})!\,,
\end{eqnarray}
where $m_{n-1}(\partial\Omega)$ denotes the $(n-1)$ dimensional Lebesgue measure of $\partial \Omega$. Next we note that
\[
m_{1}! m_{2}!\leq (m_{1}+m_{2})!\,,
\]
for all $m_1, m_2 \in \mathbb{N}$. Then we have
\[
\sup_{ {\mathrm{cl}}\Omega_{1}}|I(k,l_{k})| 
\leq 
m_{n-1}(\partial\Omega)
\|h\|_{  H^{\lambda,\rho}_{q}(G)  }
\|\varphi\|_{ C^{0}_{\omega,\rho}({\mathrm{cl}}\Omega)}
\frac{(|\beta|-1)!}{
 \rho^{|\beta|-1}
 }\,.
\]
Hence,
\begin{equation}
\label{qropo6}
\sup_{ {\mathrm{cl}}\Omega_{1}}|I|
\leq n  \rho m_{n-1}(\partial\Omega)
\|h\|_{  H^{\lambda,\rho}_{q}(G)  }
\|\varphi\|_{ C^{0}_{\omega,\rho}({\mathrm{cl}}\Omega)}
\frac{|\beta|!}{
\rho^{|\beta|}
 }\,.
\end{equation}
 On the other hand, Proposition \ref{qopoi} (iii) implies that
\begin{equation}
\label{qropo7}
\|{\mathcal{P}}^{+}[h,\partial^{\beta}\varphi]\|_{L^\infty(
\Omega)}
\leq 2^{n} \int_{\tilde{Q} }|y|^{-\lambda}\,dy
\|h\|_{A^{0}_{q,\lambda} } 
\| \partial^{\beta} \varphi\|_{L^{\infty}(\Omega)}\, .
\end{equation}
Then equality (\ref{qropo2}) and inequalities (\ref{qropo6}) and (\ref{qropo7}) imply that there exists $C>0$ such that
\[
\|\partial^{\beta}{\mathcal{P}}^{+}[h,\varphi]_{|\Omega_1}\|_{L^\infty(\Omega_1)}
\leq C \|h\|_{  H^{\lambda,\rho}(G)  }
\|   \varphi\|_{ C^{0}_{\omega,\rho}({\mathrm{cl}}\Omega)}
\frac{|\beta|!}{\rho^{|\beta| }}\qquad\forall
\beta\in {\mathbb{N}}^{n}\,,
\]
for all $(h,\varphi)\in  H^{\lambda,\rho}(G)  
\times C^{0}_{\omega,\rho}({\mathrm{cl}}\Omega)$. Hence, statement (i) holds true. 

The proof of statement (ii) follows the lines of the proof of statement (i). We only point out that formula (\ref{qropo2}) holds if we replace
${\mathcal{P}}^{+}$ by ${\mathcal{P}}^{-}$, and   the minus sign in the right hand side by a plus sign (see also (\ref{qpoder7})).   \end{proof} 

\vspace{\baselineskip}

Then, if ${\mathrm{cl}}\Omega
\subseteq
\Omega_{2}\subseteq
{\mathrm{cl}}\Omega_{2}\subseteq Q$, we have the following result in ${\mathbb{S}}[\Omega_2]^{-}$.
\begin{theorem}
\label{qropoo}
Let $\rho\in]0,+\infty[$, $\lambda\in ]0,n-1[$. Let $\Omega$ be a bounded open Lipschitz subset of ${\mathbb{R}}^{n}$ such that ${\mathrm{cl}}\Omega\subseteq Q$. 
Let $\Omega_{2}$ be an open subset of ${\mathbb{R}}^{n}$ such that 
\[
{\mathrm{cl}}\Omega
\subseteq
\Omega_{2}\subseteq
{\mathrm{cl}}\Omega_{2}\subseteq Q\,.
\]
Let
$G$ be an open relatively compact subset of ${\mathbb{R}}^{n}$ such that $({\mathrm{cl}}G)\cap q{\mathbb{Z}}^{n}=\emptyset$.

 If $\Omega\neq\emptyset$, then we assume that $G$ is an open neighborhood of the set
\[
\left\{
t-s:\, t\in {\mathrm{cl}}Q\setminus\Omega_{2}, s\in\partial\Omega
\right\}\,.
\]
 If instead $\Omega=\emptyset$, then we assume that $G=\emptyset$. 
Then the following statements hold.
\begin{enumerate}
\item[(i)] Let $\Omega\neq\emptyset$. The restriction of ${\mathcal{P}}^{+}[h,\varphi]$ to 
${\mathrm{cl}}{\mathbb{S}}[\Omega_{2}]^{-}$ belongs to $C^{0}_{q,\omega,\rho}
({\mathrm{cl}}{\mathbb{S}}[\Omega_{2}]^{-})$ for all pairs $(h,\varphi)\in H_{q}^{\lambda,\rho}(G)\times C^{0}_{\omega,\rho}
({\mathrm{cl}}\Omega)$ and the map from $H_{q}^{\lambda,\rho}(G)\times C^{0}_{\omega,\rho}({\mathrm{cl}}\Omega)$ to $C^{0}_{q,\omega,\rho}
({\mathrm{cl}}{\mathbb{S}}[\Omega_{2}]^{-})$ which takes $(h,\varphi)$ to $
{\mathcal{P}}^{+}[h,\varphi]_{|  {\mathrm{cl}}{\mathbb{S}}[\Omega_{2}]^{-}  }$
is bilinear and continuous.
\item[(ii)] The restriction of ${\mathcal{P}}^{-}[h,\varphi_{|Q\setminus{\mathrm{cl}}\Omega}
]$ to  ${\mathrm{cl}}{\mathbb{S}}[\Omega_{2}]^{-}$ belongs to $C^{0}_{q,\omega,\rho}
({\mathrm{cl}}{\mathbb{S}}[\Omega_{2}]^{-})$ for all pairs $(h,\varphi)\in H_{q}^{\lambda,\rho}(G)\times C^{0}_{q,\omega,\rho}
({\mathrm{cl}}{\mathbb{S}}[\Omega ]^{-})$ and the map from $H_{q}^{\lambda,\rho}(G)\times C^{0}_{q,\omega,\rho}({\mathrm{cl}}{\mathbb{S}}[\Omega ]^{-})$ to $C^{0}_{q,\omega,\rho}
({\mathrm{cl}}{\mathbb{S}}[\Omega_{2}]^{-})$ which takes $(h,\varphi)$ to $
{\mathcal{P}}^{-}[h,\varphi_{|Q\setminus{\mathrm{cl}}\Omega}]_{|  {\mathrm{cl}}{\mathbb{S}}[\Omega_{2}]^{-}  }$
is bilinear and continuous.  (We note that if $G=\emptyset$, then by Definition \ref{qh1r} we have $H_{q}^{\lambda,\rho}(G)=A^{1}_{q,\lambda}$.)
\end{enumerate}
\end{theorem}

\medskip

\begin{proof} We first prove statement (ii). For the sake of simplicity, we write just $\varphi$ instead of $\varphi_{|Q\setminus
{\mathrm{cl}}\Omega
}$, and we prove that if $m\in {\mathbb{N}}\setminus\{0\}$ and if $(h,\varphi)\in
H^{\lambda,\rho}_{q}(G)\times C^{m}_{q} ({\mathrm{cl}}
{\mathbb{S}}[\Omega]^{-})$, then 
$
{\mathcal{P}}^{-}[h,\varphi] _{|  {\mathrm{cl}}
{\mathbb{S}}[\Omega_{2}]^{-}
 }
\in C^{m}_{q}
({\mathrm{cl}}
{\mathbb{S}}[\Omega_{2}]^{-})$ and
\begin{eqnarray}
	\label{qropo8}
	\lefteqn{
\partial^{\beta}
{\mathcal{P}}^{-}[h,\varphi](x)=
{\mathcal{P}}^{-}[h,\partial^{\beta}\varphi](x)
}
\\
&&
\nonumber\qquad\qquad\qquad\qquad
+\sum_{k=1}^{n}\sum_{l_{k}=0}^{\beta_{k}-1}
\partial^{\beta_{n}}_{x_{n}}{...}\partial^{\beta_{k+1}}_{x_{k+1}}\partial^{l_{k}}_{x_{k}}
\biggl\{\biggr.
\int_{\partial\Omega}h(x-y)
\\
&&
\nonumber\qquad\qquad\qquad\qquad\qquad
\times
\left(\partial^{\beta_{k}-1-l_{k}}_{y_{k}}
\partial^{\beta_{k-1}}_{y_{k-1}}{...}\partial^{\beta_{1}}_{y_{1}}\varphi(y)\right)
(\nu_\Omega)_{k}(y)\,d\sigma_{y}
\biggl.\biggr\}\,,
\end{eqnarray}
for all $x\in{\mathrm{cl}}{\mathbb{S}}[\Omega_{2}]^{-}$ and for all $\beta\in {\mathbb{N}}^{n}$ such that $|\beta|\leq m$,  where we understand that 
$\sum_{l_{k}=0}^{\beta_{k}-1}$ is omitted 
if $\beta_{k}=0$, and where the integral in the right hand side is omitted if $\Omega=\emptyset$. Now let $V$ be a bounded open connected subset of ${\mathbb{R}}^{n}$ of class $C^{\infty}$ such that
\[
{\mathrm{cl}} Q\subseteq V\,,
\qquad
{\mathrm{cl}} V\cap (qz+{\mathrm{cl}}  \Omega)=\emptyset\,,
\qquad\forall z\in {\mathbb{Z}}^{n}\setminus\{0\}\,.
\]
 Possibly shrinking $V$, we can assume that
\[
 \{
 t-s:\, t\in {\mathrm{cl}} V\setminus\Omega_{2}, s\in\partial\Omega
 \}\subseteq G\,,
\]
where we understand that the set in the left hand side is empty if $\partial\Omega = \emptyset$. 
Let $W\equiv V\setminus{\mathrm{cl}}\Omega_{2}$. Since $h$ is $q$-periodic, it suffices to show that ${\mathcal{P}}^{-}[h,\varphi]_{| {\mathrm{cl}} W}\in C^{m}({\mathrm{cl}} W)$ and that (\ref{qropo8}) holds for all 
$x\in {\mathrm{cl}} W$.  If $m=1$, then the statement follows by Lemma \ref{qpoder} (ii), (iii), and (\ref{qpoderz}). 

Next we  assume that the statement holds for $m$  and we prove it for $m+1$. We will confine ourself to the case where $\Omega\neq\emptyset$. The case where $\Omega=\emptyset$ can be treated by a simplification of the argument used for $\Omega\neq\emptyset$. 

Let $(h,\varphi)\in
H^{\lambda,\rho}_{q}(G)\times C^{m+1}_{q} ({\mathrm{cl}}
{\mathbb{S}}[\Omega]^{-})$. By the inductive assumption, 
${\mathcal{P}}^{-}[h,\frac{\partial \varphi }{\partial x_{j}}]_{|{\mathrm{cl}}
{\mathbb{S}}[\Omega_{2}]^{-}}
\in  C^{m }_{q} ({\mathrm{cl}}
{\mathbb{S}}[\Omega_{2}]^{-})$ for all $j\in\{1,{...},n\}$.
Since $h\in C^{m+1}({\mathrm{cl}}G)$ and $\varphi\, , (\nu_\Omega)_{j}\in C^{0}(\partial\Omega)$, the classical differentiability theorem  for integrals depending on a parameter implies that the second term in the right hand side of formula (\ref{qpoder2}) defines a function of class $C^{m}(
 {\mathrm{cl}}W)$  (see also (\ref{qpoder7}) and following comment).  Then formula (\ref{qpoder2}) implies that $\frac{\partial}{\partial x_{j}}{\mathcal{P}}^{-}[h,\varphi]$ belongs to $C^{m}({\mathrm{cl}}W)$.
Hence, ${\mathcal{P}}^{-}[h,\varphi]_{|
{\mathrm{cl}}W
}\in  C^{m+1}(
 {\mathrm{cl}}W)$. 

Next one proves the formula for the derivatives  by following the lines  of the corresponding argument  of \cite[p.~856]{La05}: first one proves the formula for $\partial^{\beta}=\partial^{\beta_{j}}_{x_{j}}$ by finite induction on $\beta_{j}$, then one proves the formula for 
 $\partial^{\beta}=\partial^{\beta_{1}}_{x_{1}}{...}
 \partial^{\beta_{j}}_{x_{j}}$ by finite induction on $j\in\{1,{...},n\}$, and  finally one deduces that the formula holds for $|\beta|\leq m+1$.

If $(h,\varphi)\in
H^{\lambda,\rho}_{q}(G)\times C^{\infty} ({\mathrm{cl}}Q\setminus\Omega)$, then by applying the above  statement for all $m\in {\mathbb{N}}\setminus\{0\}$ we deduce that
$
{\mathcal{P}}^{-}[h,\varphi]_{|  {\mathrm{cl}}W }
$ belongs to $ C^{\infty} 
({\mathrm{cl}}W)$ and that formula (\ref{qropo8}) holds for all order derivatives.

We now assume that $(h,\varphi)\in
H^{\lambda,\rho}_{q}(G)\times C^{0}_{\omega,\rho}({\mathrm{cl}}{\mathbb{S}}[\Omega]^{-})$ and we turn to estimate the sup-norm in ${\mathrm{cl}}W$ of the sum in the right hand side of (\ref{qropo8}), which we denote by  $I$. 

We abbreviate by $I(k,l_{k})$ the $(k,l_{k})$-th term in the sum $I$, and we now estimate the supremum of $I(k,l_{k})$ in ${\mathrm{cl}}W$. We can clearly assume that $\beta_{k}>0$. Then we obtain the inequality in (\ref{qropo3}) with $W$ in place of $\Omega_{1}$. Then we can argue precisely as in the proof of Theorem \ref{qropo} by replacing $\Omega_{1}$ by $W$ and $\Omega$ by $Q\setminus{\mathrm{cl}}\Omega$ and obtain
\begin{equation}
\label{qropo12}
\sup_{ \mathrm{cl}W }|I|
\leq n \rho
m_{n-1}(\partial\Omega)
\|h\|_{  H^{\lambda,\rho}_{q}(G)  }
\|\varphi\|_{ C^{0}_{\omega,\rho}({\mathrm{cl}}Q\setminus\Omega)}
\frac{|\beta|!}{
\rho^{|\beta|}
 }\,.
\end{equation}
On the other hand, Proposition \ref{qopoi} (vi) implies that
\begin{equation}
\label{qropo13}
\|{\mathcal{P}}^{-}[h,\partial^{\beta}\varphi]\|_{L^\infty(
{\mathbb{S}}[\Omega]^{-})}
\leq 2^{n} \int_{\tilde{Q} }|y|^{-\lambda}\,dy
\|h\|_{A^{0}_{q,\lambda} } 
\| \partial^{\beta} \varphi\|_{L^{\infty}({\mathrm{cl}}Q\setminus\Omega)}\,.
\end{equation}
Then formula (\ref{qropo8}) and inequalities (\ref{qropo12}) and (\ref{qropo13}) imply that there exists $C>0$ such that
\[
\|\partial^{\beta}{\mathcal{P}}^{-}[h,\varphi]_{|W}\|_{L^\infty(W)}
\leq C \|h\|_{  H^{\lambda,\rho}_{q}(G)  }
\|   \varphi\|_{ C^{0}_{\omega,\rho}({\mathrm{cl}}W)}
\frac{|\beta|!}{\rho^{|\beta| }}\qquad\forall
\beta\in {\mathbb{N}}^{n}\,.
\]
Hence, statement (ii) holds true. 

The proof of statement (i)  follows the lines of the proof of statement (ii). We only point out that formula (\ref{qropo8}) holds if we replace
${\mathcal{P}}^{-}$ by ${\mathcal{P}}^{+}$, and   the plus sign in the right hand side by a minus sign. \end{proof}

\section{A remark on the class $H_q^{\lambda,\rho}(G)$ of weakly singular $q$-periodic kernels}\label{Sq}

In the present section,  we show that  if $q$ is a real $n\times n$ positive definite diagonal matrix (such as (\ref{diag})), then there exists  a class
of $q$-periodic kernels which is actually contained in
the class $H_{q}^{\lambda,\rho}(G)$ as in the assumptions of Theorem
 \ref{qropo} or of Theorem \ref{qropoo}
  and which is  relevant in the analysis of $q$-periodic non-homogeneous
boundary value problems for elliptic equations.   Let
$a_\alpha\in\mathbb{C}$ for all $\alpha\in\mathbb{N}^n$ with
$|\alpha|\le 2$ and let $P(x)$ be the polynomial
\[
P(x)\equiv \sum_{|\alpha|\leq 2}a_{\alpha}x^{\alpha}\qquad\forall x\in
{\mathbb{R}}^{n}\,.
\]
Then $P(D)$ denotes the partial differential operator
\[
P(D)\equiv \sum_{|\alpha|\leq 2}a_{\alpha}D^{\alpha}\,.
\]
We also assume that $P(D)$ is of  second order  and strongly elliptic, namely that
\[
\mathrm{Re}\Bigl(\sum_{|\alpha|=2}a_{\alpha}\xi^{\alpha}\Bigr)\neq 0\qquad \forall \xi\in
{\mathbb{R}}^{n}\setminus\{0\}\,.
\]
Then, a $q$-periodic  distribution $F_q$ is a $q$-periodic
fundamental solution of  $P(D)$ if
\begin{equation}
\label{introd4}
P(D) F_q=\sum_{z\in{\mathbb{Z}}^{n}}\delta_{qz}\,,
\end{equation}
where $\delta_{qz}$ denotes the Dirac measure with mass at $qz$, for
all $z\in {\mathbb{Z}}^{n}$.
Unfortunately however, not all operators $P(D)$ admit
$q$-periodic   fundamental solutions: not even well known operators,
such as the Laplace operator $\Delta$, do have one.

Instead, if we denote by $E_{2\pi i q^{-1} z}$, the function defined by
\[
E_{2\pi i q^{-1} z}(x)\equiv e^{2\pi i (q^{-1} z)\cdot x}
\qquad  \forall x\in{\mathbb{R}}^{n}\,,
\]
for all $z\in {\mathbb{Z}}^{n}$, then one can show that the set
\[
{\mathbb{Z}}(P)\equiv
\{
z\in{\mathbb{Z}}^{n}:\,P(2\pi iq^{-1}z)=0
\}
\]
is finite and that the $q$-periodic distribution
\begin{equation}
\label{introd6}
S_q\equiv \sum_{z\in
{\mathbb{Z}}^{n}\setminus {\mathbb{Z}}(P) }
\frac{1}{m_{n}(Q)}\frac{1}{P(2\pi iq^{-1}z)}E_{2\pi i q^{-1}z}
\end{equation}
satisfies the equality
\begin{equation}
\label{introd7}
P(D)S_q
=\sum_{ z\in {\mathbb{Z}}^{n} }\delta_{qz}-
\sum_{z\in {\mathbb{Z}}(P) }\frac{1}{m_{n}(Q)}E_{2\pi i q^{-1}z}\,
\end{equation}
(cf.~\textit{e.g.}, Ammari and Kang~\cite[p.~53]{AmKa07},
\cite[\S3]{LaMu11}).  Equality (\ref{introd7}) can be considered as an
effective substitute of equality (\ref{introd4}), and the distribution
 $S_q$ can be exploited  to introduce either layer or volume
potentials, which can be employed to analyze boundary value problems
on $q$-periodic domains. We note that Lin and Wang \cite{LiWa10}, Mityushev and Adler
\cite{MiAd02}, and Mamode \cite{Ma14}  have proved the validity of a
constructive formula for the $q$-periodic analog of the fundamental
solution for the Laplace equation in case $n=2$  via elliptic
functions.

As it is well known (see, {\it e.g.}, \cite[Thm.~3.5]{LaMu11}), if
$S_{q} $ is the distribution in \eqref{introd6} and if $S$ is a
classical (non-periodic) fundamental solution of the same operator $P$,
then $S_{q}-S$ is a real analytic function in $(\mathbb{R}^n \setminus
q\mathbb{Z}^n)\cup \{0\}$. Accordingly, by John \cite{Jo55}, one
deduces that
\[
\sup_{
x\in  \tilde{Q}\setminus\{0\}
}
|S_{q}(x)|\,|x|^{n-2}<+\infty\,,
\qquad
\sup_{
x\in \tilde{Q}\setminus\{0\}
}
\left|\frac{\partial S_{q}}{\partial x_{j} }(x)\right|\,|x|^{n-1}<+\infty
 \,,
\]
for all  $j\in\{1,{...},n\}$,
if $n>2$, and that
\[
\sup_{
x\in  \tilde{Q}\setminus\{0\}
}
|S_{q}(x)|\,|x|^{1/2}<+\infty\,,
\qquad
\sup_{
x\in  \tilde{Q}\setminus\{0\}
}
\left|\frac{\partial S_{q}}{\partial x_{j} }(x)\right|\,|x|^{3/2}<+\infty
\,,
\]
for all  $j\in\{1,{...},n\}$,
if $n=2$. Moreover, $S_{q}$ is analytic in ${\mathbb{R}}^{n}\setminus
q{\mathbb{Z}}^{n}$, and the classical Cauchy inequalities for the
derivatives of $S_{q}$ on  a compact set imply that
$
S_{q}\in C^{0}_{\omega,\rho}(
{\mathrm{cl}} G
)$
for all  open bounded subsets $G$ of ${\mathbb{R}}^{n}$ such that $({\mathrm{cl}}G)\cap q{\mathbb{Z}}^{n}=\emptyset$ and for  $\rho\in]0,+\infty[$ sufficiently small (cf.~\textit{e.g.},
John~\cite[p.~65]{Jo82}).   Hence,
\[
S_{q}\in H^{\max\{n-2,\frac{1}{2}\},\rho}_{q}(G)\,,
\]
and thus our class $H^{\max\{n-2,\frac{1}{2}\},\rho}_{q}(G)$ contains
the  $q$-periodic analogs of the fundamental solutions of
second order elliptic operators.

Also if we have a one-parameter analytic family of  $q$-periodic analogs of the
fundamental solution  in the space
$H^{\max\{n-2,\frac{1}{2}\},\rho}_{q}(G)$, we can apply   Theorems
\ref{qropo} and \ref{qropoo} and deduce results of analytic dependence
for the corresponding volume potentials upon the parameter and
densities (or moments) of the volume potentials  (see, {\it e.g.}, \cite{DaLaMu15}, where the authors have obtained similar results for the non-periodic case).

  \section*{Acknowledgment}
   The authors thank G.~Mishuris for fruitful discussions on subjects related to the paper.
The authors acknowledge the support of `Progetto di Ateneo: Singular perturbation problems for differential operators -- CPDA120171/12' - University of Padova and of `INdAM GNAMPA Project 2015 - Un approccio funzionale analitico per problemi di perturbazione singolare e di omogeneizzazione'. M. Dalla Riva also acknowledges the support of HORIZON 2020 MSC EF project FAANon (grant agreement MSCA-IF-2014-EF - 654795) at the University of Aberystwyth, UK. M.~Lanza de Cristoforis acknowledges the support of  the grant EP/M013545/1: ``Mathematical Analysis of Boundary-Domain Integral Equations for Nonlinear PDEs'' from the EPSRC, UK.  P.~Musolino  acknowledges the support of an `assegno di ricerca INdAM'. P.~Musolino is a S\^er CYMRU II COFUND fellow, also supported by the `S\^er Cymru National Research Network for Low Carbon, Energy and Environment'.


\begin{thebibliography}{11}


 \bibitem{Al02}
G.~Allaire, {\it Shape optimization by the homogenization method}, Springer-Verlag, New York, 2002.


\bibitem{AmKa07}
H.~Ammari and H.~Kang,  {\em Polarization and moment tensors},   Applied
  Mathematical Sciences {\bf  162},  Springer, New York, 2007.
  
\bibitem{AmKaLe09}
H.~Ammari, H.~Kang, and H.~Lee, {\em Layer potential techniques in spectral analysis}, Mathematical Surveys and Monographs  {\bf 153}, American Mathematical Society, Providence, RI, 2009.


\bibitem{BaPa89}
N.~Bakhvalov and G.~Panasenko, {\it Homogenisation: Averaging Processes in Periodic Media}, Kluwer Academic Publishers Group, Dordrecht, 1989.

\bibitem{BeLiPa78}
A.~Bensoussan, J.L.~Lions, and G.~Papanicolaou, {\it Asymptotic analysis for periodic structures}, North-Holland Publishing Co., Amsterdam-New York, 1978.



\bibitem{Ch95} 
A.~Charalambopoulos, {\em On the Fr\'echet differentiability of boundary integral operators in the inverse elastic scattering problem}, Inverse Problems {\bf 11} (1995),   1137--1161.


\bibitem{CoLe12a} 
M.~Costabel and F.~Le Lou\"er, {\em Shape derivatives of boundary integral operators in electromagnetic scattering. Part I: Shape differentiability of pseudo-homogeneous boundary integral operators}, Integral Equations Oper.~Theory {\bf 72} (2012),  509--535.

 \bibitem{CoLe12b} 
 M.~Costabel and F.~Le Lou\"er, {\em Shape derivatives of boundary integral operators in electromagnetic scattering. Part II: Application to scattering by a homogeneous dielectric obstacle}, Integral Equations 
 Oper.~Theory {\bf 73} (2012),  17--48.

 


 \bibitem{DaLa10}
 M.~Dalla Riva and M.~Lanza de Cristoforis, {\em A perturbation result for the layer potentials of general second order differential operators with constant coefficients}, J. Appl. Funct. Anal. {\bf 5} (2010),  10--30.
 
\bibitem{DaLaMu15} M.~Dalla Riva, M.~Lanza de Cristoforis, and P.~Musolino, {\em Analytic dependence of volume potentials corresponding to parametric families of fundamental solutions}, Integral Equations Operator Theory, {\bf 82} (2015), 371--393.


 
  \bibitem{DaMu13} M.~Dalla Riva and P.~Musolino, {\em A singularly perturbed non-ideal transmission problem and application to the effective conductivity of a periodic composite}, SIAM J. Appl. Math. {\bf 73} (2013), 24--46.
 
 
 \bibitem{De85}
K.~Deimling, {\em Nonlinear functional analysis},
  Springer-Verlag, Berlin, 1985.




 
 \bibitem{GiTr83}
D.~Gilbarg and N.S.~Trudinger, {\em Elliptic partial 
differential equations of second order},   Springer Verlag, 1983.

\bibitem{HaKr04}
H.~Haddar and R.~Kress, {\em On the Fr\'echet derivative for obstacle scattering with an impedance boundary condition}, SIAM J. Appl. Math. {\bf 65} (2004),  194--208. 



\bibitem{He95} 
F.~Hettlich, {\em Fr\'echet derivatives in inverse obstacle scattering}, Inverse Problems {\bf 11} (1995),   371--382. 

\bibitem{Jo55}
F.~John, {\em Plane waves and spherical means applied to partial differential
equations}, Interscience Publishers, New York-London, 1955.


\bibitem{Jo82}
F.~John, {\em Partial Differential Equations}, Springer, New York, {\it etc.}, 1982.


\bibitem{KaMiPe15}
D.~Kapanadze, G.~Mishuris, and E.~Pesetskaya, {\em Exact solution of a nonlinear heat conduction problem in a doubly periodic 2D composite material}, Arch. Mech. (Arch. Mech. Stos.) {\bf 67} (2015), 157--178. 

\bibitem{KaMiPe15b}
D.~Kapanadze, G.~Mishuris, and E.~Pesetskaya, {\em Improved algorithm for analytical solution of the heat conduction problem in doubly periodic 2D composite materials}, Complex Var. Elliptic Equ. {\bf 60} (2015), 1--23. 

\bibitem{Ki93} 
A.~Kirsch, {\em The domain derivative and two applications in inverse scattering theory}, Inverse Problems {\bf 9} (1993),  81--96.

\bibitem{KrPa99}
R.~Kress and L.~P\"aiv\"arinta, {\em On the far field in obstacle scattering},
SIAM J. Appl. Math. {\bf 59}  (1999),  1413--1426. 

\bibitem{La05}
M.~Lanza de Cristoforis, {\em  A domain perturbation problem for the Poisson equation}, Complex Var. Theory Appl. {\bf  50} (2005),   851--867.

\bibitem{LaMu11} 
M.~Lanza de Cristoforis and P.~Musolino, {\em A perturbation result for periodic layer potentials of general second order differential operators with constant coefficients}, Far East J. Math. Sci. (FJMS) {\bf 52} (2011), 
  75--120.
  
 \bibitem{LaMu14} M.~Lanza de Cristoforis and P.~Musolino, {\em A quasi-linear heat transmission problem in a periodic two-phase dilute composite. A functional analytic approach}, Comm. Pure Appl. Anal.  {\bf13} (2014), 2509--2542.

\bibitem{LaRo04}
M.~Lanza~de~Cristoforis and L.~Rossi, {\em 
Real analytic dependence of simple and double 
layer potentials upon perturbation 
of the support and of the density}, J. Integral Equations 
Appl.  {\bf 16} (2004),  137--174.

\bibitem{LaRo08}
M.~Lanza de Cristoforis and L.~Rossi, {\em Real analytic dependence of simple and double layer potentials for the Helmholtz equation upon perturbation of the support and of the density}, Analytic methods of analysis and differential equations: AMADE 2006, 193--220, Camb. Sci. Publ., Cambridge, 2008.



\bibitem{Le12}
F.~Le Lou\"er, {\em On the Fr\'echet derivative in elastic obstacle scattering}, SIAM J. Appl. Math. {\bf 72} (2012),  1493--1507.

\bibitem{LiWa10}
{C.S.~Lin and C.L.~Wang},  {\em Elliptic functions, Green functions and the mean field equations on tori}, {Ann. of Math. } {\bf 172} (2010),  911--954.

\bibitem{Ma14} M.~Mamode,  {\em Fundamental solution of the Laplacian on flat tori and boundary value problems for the planar Poisson equation in rectangles}, Bound.~Value Probl. {\bf 2014} (2014), 221, 9 pp.


  \bibitem{Mi02} 
G.W.~Milton, {\it The Theory of Composites,} Cambridge Monographs on Applied and Computational Mathematics, Cambridge University Press, Cambridge, 2002.

\bibitem{MiAd02}  V.~Mityushev and P.M.~Adler.  {\em Longitudinal permeability of spatially periodic rectangular arrays of circular cylinders. I. A single cylinder in the unit cell}, ZAMM Z. Angew. Math. Mech., {\bf 82}, (2002), pp.~335--345.



\bibitem{MiSl14}
G.S.~Mishuris and L.I.~Slepyan, {\em Brittle fracture in a periodic structure with internal potential energy}, Proc.~R.~Soc.~Lond.~Ser.~A Math.~Phys.~Eng.~Sci. {\bf 470} (2014), 20130821, 25 pp.

\bibitem{MiPeRo08}
V.V.~Mityushev, E.~Pesetskaya, and S.V.~Rogosin, Analytical Methods for Heat Conduction in Composites and Porous Media, in {\it Cellular and Porous Materials: Thermal Properties Simulation and Prediction} (eds A. \"Ochsner, G.E. Murch and M.J.S. de Lemos), Wiley-VCH Verlag GmbH \& Co. KGaA, Weinheim, Germany, 2008.



\bibitem{MoMoPo02}
A.B.~Movchan, N.V.~Movchan,  and C.G.~Poulton, {\em Asymptotic models of fields in dilute and densely packed composites}, Imperial College Press, 2002.


\bibitem{Ne67}
J.~Ne\v cas, {\em Les m\'{e}thodes directes en th\'{e}orie des \'{e}quations elliptiques}, Masson et Cie,  Paris 1967.

\bibitem{Po94}
R.~Potthast, {\em Fr\'echet differentiability of boundary integral operators in inverse acoustic scattering}. Inverse Problems {\bf 10} (1994),  431--447. 


\bibitem{Po96a} 
R.~Potthast, {\em Fr\'echet differentiability of the solution to the acoustic Neumann scattering problem with respect to the domain}, J. Inverse Ill-Posed Probl. {\bf 4} (1996),  67--84.
 
\bibitem{Po96b}  
R.~Potthast, {\em Domain derivatives in electromagnetic scattering}, Math. Methods Appl. Sci. {\bf 19} (1996),  1157--1175. 

\bibitem{Pr98}
L.~Preciso, {\em Perturbation analysis of the conformal sewing problem and
  related problems}, PhD Dissertation, University of Padova, 1998.
  
  \bibitem{Pr00} 
L.~Preciso, {\em Regularity of the composition and of the inversion 
operator and perturbation analysis of the conformal sewing problem in 
Roumieu type spaces},  Nat.\ Acad.\ Sci.\ Belarus, Proc.\ Inst.\ Math.  
{\bf 5} (2000),  99--104.



\bibitem{Tr87}
G.M.~Troianiello, {\em Elliptic differential equations 
and obstacle problems}, Plenum Press, New York and London, 1987.



\end{thebibliography}
\end{document}